\documentclass{amsart}
\usepackage{amssymb}
\usepackage{mathrsfs}
\usepackage[shortalphabetic]{amsrefs}
\usepackage{array}
\usepackage{footnote}
\usepackage{ulem}
\usepackage{amsmath}
\usepackage{bbm}
\usepackage{lscape}
\usepackage[shortlabels]{enumitem}
\usepackage{caption}
\usepackage{tikz-cd}
\usepackage{multicol}

\makeatletter
\tikzcdset{
    cong/.style={"\cong" {sloped, description, yshift=0pt,#1}, phantom},
    simeq/.style={"\simeq" {sloped, description, yshift=0pt,#1}, phantom},
    snake/.style={
        out= east, in=west,
        to path={
            \pgfextra{
                \pgfextractx{\pgf@xa}{\pgfpointanchor{\tikztostart}{east}}
                \pgfextractx{\pgf@xb}{\pgfpointanchor{\tikztotarget}{west}}
                \pgfextracty{\pgf@ya}{\pgfpointanchor{\tikztostart}{center}}
                \pgfextracty{\pgf@yb}{\pgfpointanchor{\tikztotarget}{center}}
                \edef\tikzstartx{\the\pgf@xa}
                \edef\tikzendx{\the\pgf@xb}
                \edef\midy{\the\dimexpr0.5\dimexpr\pgf@ya\relax +0.5\dimexpr\pgf@yb\relax}
            }
            to[in=0,out=180,looseness=0.5] (\tikzstartx,\midy)
            -| ([xshift=-2ex]\tikztotarget.west)
            -- (\tikztotarget)}
    }
}

\makeatother
\usepackage{tikzcdintertext}

\usepackage[all]{xy}
\usepackage{todonotes}
\usepackage{upgreek}
\usepackage{verbatim}

\usepackage{multicol}
\usepackage{wrapfig}

\usepackage{longtable,booktabs}

\usepackage{hyperref}  
\hypersetup{%
  bookmarksnumbered=true,%
  bookmarks=true,%
  colorlinks=true,%
  linkcolor=blue,%
  citecolor=blue,%
  filecolor=blue,%
  menucolor=blue,%
  urlcolor=blue,%
  pdfnewwindow=true,%
  pdfstartview=FitBH}

\newcommand{\cA}{\mathcal{A}}
\newcommand{\cB}{\mathcal{B}}

\newcommand {\M}{\mathrm{M}}

\newcommand{\C}{\mathbb{C}}

\newcommand{\bE}{\mathbb{E}}

\newcommand{\F}{\mathbb{F}_2}

\newcommand{\bM}{\mathbb{M}_2}

\newcommand{\R}{\mathbb{R}}
\newcommand{\bS}{\mathbb{S}}

\newcommand{\Z}{\mathbb{Z}}

\newcommand{\gr}{\mathrm{gr}}

\newcommand{\rH}{\mathrm{H}}
\newcommand{\rHR}{\rH_\R}
\newcommand{\rHC}{\rH_\Ctwo}

\newcommand{\Top}{\mathbf{Top}}

\newcommand{\xrtarr}[1]{\stackrel{#1}{\longrightarrow}}
\newcommand{\rtarr}{\longrightarrow}

\newcommand{\iso}{\cong}
\newcommand{\sma}{\wedge}

\newcommand{\hsf}{\mathsf{h}}

\newcommand{\sfh}{{\sf h}}
\newcommand{\may}{{\sf m}}
\newcommand{\GFixFunctor}{\Phi}
\newcommand{\GFix}[1]{\GFixFunctor(#1)}
\newcommand{\CFix}[1]{{#1}^{\Ctwo}}
\newcommand{\Fix}[1]{{#1}^{\Ctwo}}
\newcommand{\Res}{\Re}
\newcommand{\ParRes}[1]{\Res(#1)}
\newcommand{\Sp}{\mathbf{Sp}} 
\newcommand{\SpfinR}{\mathbf{Sp}_{2, \mr{fin}}^\R}
\newcommand{\SpfinC}{\mathbf{Sp}_{2, \mr{fin}}^{\mr{C}_2}}
\newcommand{\RP}{\mathbb{RP}}
\newcommand{\CP}{\mathbb{CP}}

\newcommand{\RO}{\mathrm{RO}}
\newcommand{\rE}{\mathrm{E}}
\newcommand{\rX}{\mathrm{X}}
\newcommand{\uF}{\underline{\mathbb{F}}_2}
\newcommand{\mr}[1]{\mathrm{#1}}
\newcommand{\eqHF}{\mr{H}\underline{\mathbb{F}}_2}
\newcommand{\RHF}{\mr{H}_\R\F}

\newcommand{\Ctwo}{{\mr{C}_2}}
\newcommand{\nil}{\mr{nil}}
\newcommand{\yy}{{\sf y}}
\newcommand{\xx}{{\sf x}}
\newcommand{\ttt}{{\sf t}}
 
\newcommand{\Y}{\mathcal{Y}}

\renewcommand{\emph}[1]{\textit{#1}}
\newcommand{\Pow}{\mathcal{P}}
\newcommand{\uu}{{\sf u}}

\newcommand{\modPhi}{\hat{\Phi}_*}

\newcommand{\Young}[6]{
\begin{tikzpicture}[scale =.2]
\draw  (0,0) -- (0,3) -- (3,3) -- (3,2) -- (2,2) -- (2,1) -- (1,1) -- (1,0) -- (0,0);
\draw (0,1) -- (1,1) -- (1,2) -- (2,2) -- (2,3);
\draw (0,2) -- (1,2) -- (1,3);
\node[] at (.5,.5) {{\tiny #6}};
\node[] at (.5,1.5) {{\tiny #4}};
\node[] at (.5,2.5) {{\tiny #1}};
\node[] at (1.5,1.5) {{\tiny #5}};
\node[] at (2.5,2.5) {{\tiny #3}};
\node[] at (1.5,2.5) {{\tiny#2}};
\end{tikzpicture}
} 
\newcommand{\cYoung}[6]{
\begin{tikzpicture}[scale =.2]
\draw  (0,0) -- (0,3) -- (3,3) -- (3,2) -- (2,2) -- (2,1) -- (1,1) -- (1,0) -- (0,0);
\draw (0,1) -- (1,1) -- (1,2) -- (2,2) -- (2,3);
\draw (0,2) -- (1,2) -- (1,3);
\node[] at (.5,.5) {{\tiny #6}};
\node[] at (.5,1.5) {{\tiny #4}};
\node[] at (.5,2.5) {{\tiny #1}};
\node[] at (1.5,1.5) {{\tiny #5}};
\node[] at (2.5,2.5) {{\tiny #3}};
\node[] at (1.5,2.5) {{\tiny #2}};
\draw[color=red] (-.2,3.2) -- (3,0);
\end{tikzpicture}
}

\newcommand{\sfR}{{\sf R}}
\newcommand{\sfC}{{\sf C}}

\DeclareMathOperator{\Ext}{Ext}

\newcommand{\bMR}{\bM^\R}
\newcommand{\bMC}{\bM^{\Ctwo}}
\newcommand{\cAR}{\cA^\R}
\newcommand{\cAC}{\cA^{\Ctwo}}

\newcommand{\Sq}{\mathrm{Sq}}
\newcommand{\eSq}{\underline{\mathrm{Sq}}}

\newif\ifCharts    
\Chartstrue

\numberwithin{equation}{section}
\numberwithin{figure}{section}

\def\makeautorefname#1#2{\expandafter\def\csname#1autorefname\endcsname{#2}}
\makeautorefname{theorem}{Theorem}%
\makeautorefname{thm}{Theorem}%
\makeautorefname{prop}{Proposition}%
\makeautorefname{defn}{Definition}%
\makeautorefname{addm}{Addendum}%
\makeautorefname{mainthm}{Main theorem}%
\makeautorefname{corollary}{Corollary}%
\makeautorefname{const}{Construction}%
\makeautorefname{cor}{Corollary}%
\makeautorefname{lemma}{Lemma}%
\makeautorefname{notn}{Notation}%
\makeautorefname{rmk}{Remark}%
\makeautorefname{lem}{Lemma}%
\makeautorefname{probl}{Problem}%
\makeautorefname{section}{Section}%
\makeautorefname{subsection}{Subsection}%
\makeautorefname{sublemma}{Sublemma}%
\makeautorefname{sublem}{Sublemma}%
\makeautorefname{subl}{Sublemma}%
\makeautorefname{prop}{Proposition}%
\makeautorefname{warn}{Warning}%
\makeautorefname{eg}{Example}%
\makeautorefname{figure}{Figure}%
\makeautorefname{qst}{Question}%

\newtheorem{thm}{Theorem}[section]
\newtheorem{cor}{Corollary}[section]
\newtheorem{prop}{Proposition}[section]
\newtheorem{lem}{Lemma}[section]

\newtheorem{conj}{Conjecture}[section]

\theoremstyle{definition}
\newtheorem{defn}{Definition}[section]

\newtheorem{eg}{Example}[section]
\newtheorem{notn}{Notation}[section]

\newtheorem{rmk}{Remark}[section]

\newenvironment{pf}{\begin{proof}}{\end{proof}}

\makeatletter
\let\c@cor=\c@thm
\let\c@con=\c@thm
\let\c@prop=\c@thm
\let\c@lem=\c@thm
\let\c@conj=\c@thm
\let\c@defn=\c@thm
\let\c@eg=\c@thm
\let\c@notn=\c@thm
\let\c@rmk=\c@thm
\let\c@warn=\c@thm
\let\c@qst=\c@thm
\let\c@probl=\c@thm
\let\c@equation=\c@thm
\let\c@figure=\c@thm
\makeatother

\usepackage[scr=boondoxo]{mathalpha}
\newcommand{\jay}{\mathscr{j}}

\newcolumntype{L}{>{$}l<{$}}

\newenvironment{psmallmatrix}
  {\left(\begin{smallmatrix}}
  {\end{smallmatrix}\right)}

\newcounter{themyfigure}

\definecolor{mycolor}{RGB}{100, 85, 152}

\title[On  realizations of $\cAR(1)$]{On  realizations of the subalgebra $\mathcal{A}^{\R}(1)$ of the $\R$-motivic Steenrod Algebra}

\author{P. Bhattacharya}
\address{Department of Mathematics,
University of Notre Dame,
Notre Dame, IN  46556} 
\email{pbhattac@nd.edu}
\author{B. Guillou}
\address{Department of Mathematics, The University of Kentucky, Lexington, KY 40506--0027}
\email{bertguillou@uky.edu}
\author{A. Li}
\address{Department of Mathematics, The University of Kentucky, Lexington, KY 40506--0027}
\email{ang.li1414201@uky.edu}

\thanks{Bhattacharya was supported by NSF grant DMS-2005476. Guillou and Li were  supported by NSF grant DMS-2003204.}

\begin{document}

\begin{abstract} In this paper, we show that the finite subalgebra $\cAR(1)$, generated by $\Sq^1$ and $\Sq^2$, of the $\R$-motivic Steenrod algebra $\cAR$ can be given 128 different $\cAR$-module structures. We also show that all of these$\cA$-modules can be realized as the cohomology of a $2$-local finite $\R$-motivic spectrum. The realization results are obtained using an $\R$-motivic analogue of the Toda realization theorem. We notice that each realization of $\cAR(1)$ can be expressed as a cofiber of an $\R$-motivic $v_1$-self-map. The  $\Ctwo$-equivariant analogue of the above results then follows because of the Betti realization functor. We identify a relationship between the $\RO(\Ctwo)$-graded Steenrod operations on a $\Ctwo$-equivariant space and the classical Steenrod operations on both its underlying space and its fixed-points. This technique is then used to identify the geometric fixed-point spectra of the $\Ctwo$-equivariant realizations of $\cA^\Ctwo(1)$. We find another application of the $\R$-motivic Toda realization theorem:  we produce an $\R$-motivic, and consequently a $\Ctwo$-equivariant, analogue of the Bhattacharya-Egger spectrum $\mathcal{Z}$, which could be of independent interest. 
\end{abstract}

\maketitle

\dedicatory{
The first author would like to dedicate this work to his baba Jayanta Bhattacharya whose life was lost to covid 19 on May 2, 2021. He always said, \emph{whatever you do do it in depth, and leave no stone unturned.} You will be deeply missed!} 

\tableofcontents
  
 \section{Introduction}
This paper is a continuation of the work that began in \cite{BGL}, where we studied periodic self-maps of a certain finite $\R$-motivic spectrum $\Y^{\R}_{(\hsf, 1)}$. There, we proved that $\Y^{\R}_{(\hsf, 1)}$ supports a $1$-periodic $v_{(1, \nil)}$-self-map (see \cite{BGL}*{Definition 1.7})
\begin{equation} \label{prevselfmap}
\begin{tikzcd}
v: \Sigma^{2,1}\Y^\R_{({\hsf}, 1)} \rar & \Y^\R_{({\hsf}, 1)},
\end{tikzcd}
\end{equation}
whose cofiber realizes the sub-algebra $\cAR(1)$ of the $\R$-motivic Steenrod algebra $\cAR$ generated by $\Sq^1$ and $\Sq^2$.

 The spectrum $\Y^{\R}_{({\sf h}, 1)}$ is 
 an $\R$-motivic lift of the  classical  spectrum \[ \Y:= \Sigma^{-3} \CP^2 \sma \RP^2.\] From the chromatic point of view, the spectrum $\Y$ is extremely useful because it supports a $v_1$-self-map of lowest possible periodicity, that is, one. Famously, Mark Mahowald used the spectrum $\Y$ and the low periodicity of its $v_1$-self-map to prove the height $1$ telescope conjecture at the  prime $2$ \cites{M81, M82}. However, $1$-periodic $v_1$-self-maps of $\Y$ are not unique. In fact, up to homotopy, there are eight different  $v_1$-self-maps  supported by $\Y$, all of whose cofibers are realizations of $\cA(1)$ (see \cite{DM}). Up to weak equivalence, there are four different finite spectra realizing $\cA(1)$, and all of them can be  obtained as the cofiber of some $v_1$-self-map of $\Y$.  These four different realizations can be distinguished by their $\cA$-module structures.  Therefore, it is natural to ask if all of the $v_1$-self-maps of $\Y$ can be lifted to $\R$-motivic analogues, and whether all of the $\R$-motivic realizations of $\cAR(1)$ can be obtained as the cofiber of such a lift. 

The  answer to the above question is complicated by the fact that there are multiple $\R$-motivic lifts of the spectrum $\Y$ (see \cite{BGL}). Even if we insist on those lifts which can potentially realize $\cAR(1)$ as a cofiber of a periodic self-map, we are left with two choices; $\Y^\R_{({\hsf}, 1)}$ and $\Y^\R_{(2, 1)}$. We state our first result towards answering these questions after establishing some notations. 
Further, we shall see that some realizations of $\cAR(1)$ must be given as the cofiber of a map between $\Y^\R_{(\hsf,1)}$ and $\Y^\R_{(2,1)}$ rather than as the cofiber of a self-map of either. 

Before describing the results of this article, we present some notation that will be used throughout.

\begin{notn}\label{NotationList} Throughout this paper, we use the following notations: 
\begin{itemize}
\item $\Sp^\R$ -- the  $\infty$-category of $\R$-motivic spectra. \vspace{2pt}
\item $\Sp^{\Ctwo}$ --  the  $\infty$-category of genuine $\Ctwo$-equivariant spectra.  \vspace{2pt}
\item $\RHF$  -- the $\R$-motivic Eilenberg-Mac~Lane spectrum with $\F$-coefficients.  \vspace{2pt}
\item $\eqHF$ -- the $\mr{C}_2$-equivarient Eilenberg-Mac~Lane spectrum at the constant Mackey functor $\underline{\mathbb{F}}_2$.  \vspace{2pt}
\item $\SpfinR$  -- the category of cellular $\RHF$-complete $\R$-motivic spectra with finitely many cells.  \vspace{2pt}
\item We denote the $1$-dimensional trivial $\R$-representation of $\Ctwo$ by $\upepsilon$, the sign representation by $\upsigma$ and the regular representation by $\uprho$. \vspace{2pt}
\item 
$\rHR^{n,m}(\mr{E}):= [\mr{E}, \Sigma^{n,m}\RHF ]$ -- the $\R$-motivic cohomology of $\mr{E} \in \Sp^{\R}$ with constant sheaf $\F$, where $n$ is the topological degree and $m$ is the motivic weight. \vspace{2pt}
\item $\SpfinC$  -- the category of cellular $\eqHF$-complete $\mr{C}_2$-equivariant spectra with finitely many cells. \vspace{2pt}
\item 
$\rHC^{\star}\!(\mr{E})= [\mr{E}, \rH\uF]_{-\star}^{\Ctwo}$ -- the $\RO(\Ctwo)$-graded cohomology of $\mr{E} \in  \Sp^{\Ctwo}$ with coefficients in the constant Mackey functor. We will often use motivic bigrading for  $\rHC^{\star}\!(\mr{E})$ under the identification 
\[ (n,m) \rightsquigarrow (n-m) \upepsilon + m \upsigma.\]
\item $\bMR := \pi_{\ast, \ast}\RHF$. By a calculation of  Voevodsky \cite{V} $$\bMR \cong \F[\tau, \rho]$$ with  $|\tau| = (0,-1)$ and $|\rho | = (-1,-1)$.  \vspace{2pt}
\item $\bMC := \pi_{\star}\eqHF$. This computation can be found in \cite{Caruso}*{Appendix} or \cite{HK}*{Proposition~6.2} and is given by
\[ \bMC \cong \F[u_{\upsigma}, a_{\upsigma}] \oplus \Theta\{ u_{\upsigma}^{-i} a_{\upsigma}^{-j}: i,j \geq 0\},\]
 where  $|u_{\upsigma}| = (0,-1)$, $|a_{\upsigma}| = (-1, -1)$ and  $|\Theta| = (0,2)$. \vspace{2pt}. 
 \item 
We follow \cites{LowMilnorWitt, BI, BGL} in grading $\Ext_{\cAR}$ as $\Ext^{s, f, w}_{\cAR}$, where $s$ is the stem, $f$ is the Adams filtration, and $w$ is the motivic weight.  \vspace{2pt}
\end{itemize}
\end{notn}

Our first result concerns realizations of $\cAR(1)$.

\begin{thm} \label{thm:128R} There exists $128$ different $\cAR$-modules whose underlying $\cAR(1)$-module structures are free on one generator, all of which can be realized as $\rHR^{*, *}(\mr{X})$ for some $\mr{X} \in \SpfinR$.
\end{thm}

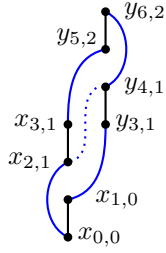
\begin{figure}[h]
 \[
\begin{tikzpicture}\begin{scope}[thick, every node/.style={sloped,allow upside down}, scale=0.5]
\draw (0,0)  node[inner sep=0] (v00) {} -- (0,1) node[inner sep=0] (v01) {};
\draw (0,2)  node[inner sep=0] (v11) {} -- (0,3) node[inner sep=0] (v12) {};
\draw (1,3)  node[inner sep=0] (v22) {} -- (1,4) node[inner sep=0] (v23) {};
\draw (1,5)  node[inner sep=0] (v33) {} -- (1,6) node[inner sep=0] (v34) {};
 \draw [color=blue] (v00) to [out=150,in=-150] (v11);
 \draw [color=blue] (v01) to [out=15,in=-90] (v22);
 \draw [color=blue] (v12) to [out=90,in=-165] (v33);
 \draw [color=blue] (v23) to [out=30,in=-30] (v34);
 \draw [dotted][color=blue] (v11) to [out=15,in=-150] (.8, 4);
\filldraw (v00) circle (2.5pt);
\filldraw (v01) circle (2.5pt);
\filldraw (v11) circle (2.5pt);
\filldraw (v12) circle (2.5pt);
\filldraw (v22) circle (2.5pt);
\filldraw (v23) circle (2.5pt);
\filldraw (v33) circle (2.5pt);
\filldraw (v34) circle (2.5pt);
\draw (0,0) node[right]{$ x_{0,0}$} (0,1) node[right]{ \ \ $x_{1,0}$} (0,2) node[left,xshift={1pt}]{$x_{2,1}  \ $ } (0,3) node[left]{$x_{3,1} $  };
\draw (1,6) node[right]{$\ y_{6,2}$} (1.3,5.3) node[left]{$y_{5,2}  \ $} (1,4) node[right]{$ \ y_{4,1} $} (1,3) node[right]{$ y_{3,1}$};
\end{scope}\end{tikzpicture}
\]
\caption{ We depict a singly-generated free $\cAR(1)$-module,  where  each $\bullet$ represents a $\bM^\R$-generator. The black and blue lines represent the action of  motivic $\Sq^1$ and $\Sq^2$, respectively. A dotted line represents that the action hits the $\tau$-multiple of the given $\mathbb{M}_2^\R$-generator. }
\label{fig:motivicA1}
 \end{figure}
 
\begin{notn} 
For the rest of the paper, we fix an $\bM^\R$-basis 
\[ \{ x_{0,0}, x_{1,0}, x_{2,1}, x_{3,1},y_{3,1}, y_{4,1}, y_{5,2}, y_{6,2} \} \]
 of $\cAR(1)$ {as in \autoref{fig:motivicA1}}, so that 
\begin{multicols}{3}
\begin{itemize}
\item $\Sq^1(x_{0,0}) = x_{1,0}$
\item $\Sq^1(x_{2,1}) = x_{3,1}$
\item $\Sq^1(y_{3,1}) = y_{4,1}$
\end{itemize}
\columnbreak
\begin{itemize}
\item $\Sq^1(y_{5,2}) = y_{6,2}$
\item $\Sq^2(x_{0,0}) = x_{2,1}$
\item $\Sq^2(x_{1,0}) = y_{3,1}$
\end{itemize}
\columnbreak
\begin{itemize}
\item $\Sq^2(x_{2,1}) = \tau y_{4,1}$
\item $\Sq^2(x_{3,1}) = y_{5,2}$
\item $\Sq^{2}(y_{4,1}) = y_{6,2}$.
\end{itemize}
\end{multicols}
\end{notn}

We now record all $128$ $\cAR$-modules of \autoref{thm:128R}  using the basis above.

\begin{thm}  \label{list}
For every vector  $ ( \alpha_{03}, \beta_{03},\beta_{14}, \beta_{06}, \beta_{25}, \beta_{26}, \gamma_{36}) \in \mathcal{V} = \F^7$ and 
\[\jay_{24}= \beta_{03}  \gamma_{36} + \alpha_{03}(\beta_{25} + \beta_{26}),\]
 there exists a unique isomorphism class of $\cAR$-module structures on $\cAR(1)$ determined by the formulas
\begin{enumerate}[(i)]
\item $\Sq^4(x_{0,0}) = \beta_{03} ( \rho \cdot   y_{3,1}) +  (1+ \beta_{03} + \beta_{14}) ( \tau \cdot   y_{4,1}) + \alpha_{03}( \rho \cdot  x_{3,1})$
\item $\Sq^{4}(x_{1,0}) = y_{5,2} + \beta_{14} ( \rho \cdot   y_{4,1})  $ 
\item $\Sq^4(x_{2,1}) = \beta_{26}(\tau \cdot y_{6,2}) + \beta_{25} (\rho \cdot y_{5,2}) + \jay_{24} (\rho^2 \cdot y_{4,1})$
\item $\Sq^4(x_{3,1}) = (\beta_{25} + \beta_{26})(\rho \cdot y_{6,2})$
\item $\Sq^4(y_{3,1})= \gamma_{36}(\rho \cdot y_{6,2})$
\item $\Sq^8(x_{0,0}) = \beta_{06} (\rho^2 \cdot y_{6,2})$.
\end{enumerate}
Further, any $\cAR$-module whose underlying $\cAR(1)$-module is free on one generator is isomorphic to one listed above. 
\end{thm}

 \begin{notn} For any vector $\overline{\mr{v}} \in \mathcal{V}$, we denote the corresponding $\cAR$-module in  $\autoref{list}$ by $\cAR_{ \overline{\mr{v}}}(1)$. 
By $\cAR_{1}[\overline{\mr{v}}]$, 
 we denote an object of $\Sp^\R_{2,\mr{fin}}$, whose cohomology is isomorphic to $\cAR_{ \overline{\mr{v}}}(1)$ as an $\cAR$-module. We let 
 \[ \cAR_1 := \{ \cAR_{1} [\overline{\mr{v}}]: \overline{\mr{v}} \in \mathcal{V}\}/(\text{weak equivalence}) \]
 denote the set of equivalence classes of finite $\R$-motivic spectra whose cohomology are free of rank 1 over $\cAR(1)$.
 \end{notn}
 
Let $\mathcal{B}_{{\sf h}}^\R(1)$ and $\mathcal{B}_{2}^{\R}(1)$ denote the $\cA^{\R}$-modules $\rHR^{*,*}(\Y^\R_{({\sf h}, 1)})$ and $\rHR^{*,*}(\Y^\R_{(2, 1)})$, respectively.  As shown in \cite{BGL}*{Lemma~4.6}, these differ in that the bottom cell of $\Y^\R_{(2, 1)}$ supports a $\Sq^4$, whereas the bottom cell of $\Y^\R_{({\sf h}, 1)}$ does not.
In \cite{BGL}, we  used a  method due to Smith (\cite{Rav}*{Appendix C}) to produce the $\cAR$-module $\cAR_{1}[{\bf 0}]$. Then 
{we} observed that $\cAR_{1}[{\bf 0}]$ fits into a short exact sequence  
\[ 
\begin{tikzcd}
 \Sigma^{3,1} \mathcal{B}_{{\sf h}}^\R(1) \rar[hook] &\cAR_{1}[{\bf 0}] \rar[two heads]  & \mathcal{B}_{{\sf h}}^\R(1) 
\end{tikzcd}
\]
that can be realized as a cofiber sequence of finite spectra. The connecting map of this cofiber sequence is the map \eqref{prevselfmap}. In this paper, we extend the above result of \cite{BGL} to prove the following. 

\begin{thm}\label{thm:exactcofiber}
Given  $\overline{\mr{v}} = ( \alpha_{03}, \beta_{03},\beta_{14}, \beta_{06}, \beta_{25}, \beta_{26}, \gamma_{36})  \in \mathcal{V}$, define 
\[ \epsilon = \begin{cases} 
{\sf h} &  \text{if\ } \beta_{25} + \beta_{26} + \gamma_{36} = 0 \\
2 & \text{if\ } \beta_{25} + \beta_{26} + \gamma_{36} = 1,
\end{cases} 
\qquad \text{and} \qquad
 \delta = \begin{cases}
{\sf h} & \text{if\ }  \alpha_{03} + \beta_{03}= 0 \\
2 & \text{if\ } \alpha_{03} + \beta_{03} = 1.
\end{cases} \]
Then there  exists a short exact sequence 
\begin{equation} \label{exactBAB}
\begin{tikzcd}
 \Sigma^{3,1} \mathcal{B}_{\epsilon}^\R(1) \rar[hook] &\cAR_{\overline{\mr{v}}}(1) \rar[two heads]  & \cB_{\delta}^\R(1)
 \end{tikzcd}
 \end{equation}
 of $\cAR$-modules. Moreover, this exact sequence can be realized as the cohomology of a cofiber sequence 
\begin{equation}
\label{cofiberYA1}
\begin{tikzcd}
\mathcal{Y}_{(\delta, 1)}^\R  \rar &\cAR_{1}[\overline{\mr{v}}] \rar & \Sigma^{3,1} \mathcal{Y}_{(\epsilon , 1)} ^\R
 \end{tikzcd}
\end{equation}
 in the category $\SpfinR$. 
\end{thm}

The map of spectra that underlies the  connecting map 
\begin{equation} \label{v1nil}
\begin{tikzcd}
v\colon \Sigma^{2,1} \mathcal{Y}_{(\epsilon , 1)}^\R \rar & \mathcal{Y}_{(\delta, 1)}^\R 
 \end{tikzcd}
\end{equation}
of \eqref{cofiberYA1} 
is a   $v_1$-self-map of $\Y$ of periodicity $1$.

The algebraic part of \autoref{thm:exactcofiber} is a straightforward consequence of \autoref{list} once we identify the $\cAR$-modules $\mathcal{B}_{{\sf h}}^\R(1)$ and $\mathcal{B}_{2}^\R( 1)$ (we refer to \cite{BGL}*{Figure 4.7} for a complete description). However, the topological assertions in \autoref{thm:exactcofiber}, as well as  in \autoref{thm:128R}, require a technical tool, which we refer to as the  \emph{$\R$-motivic Toda realization theorem.}

\begin{thm}[$\R$-motivic Toda realization theorem] \label{TRT} Let $\mr{M}$ be an $\cAR$-module whose underlying $\bMR$-module is free and finite. There exists $\rX \in \SpfinR$  such that $\rHR^{*,*}(\rX) \cong \mr{M}$ as $\cAR$-modules if 
\begin{equation}\label{Obgroup}
 \Ext^{-2, f, 0}_{\cAR}(\mr{M}, \mr{M}) = 0
 \end{equation}
for all  $f \geq 3$. 
\end{thm}  

In this paper,  we also prove various weaker versions of the $\R$-motivic Toda realization theorem (see \autoref{wTRT1}, \autoref{wTRT2} and \autoref{wTRT3}), which are perhaps more convenient for application purposes. 

A realization theorem is often accompanied by a uniqueness theorem, as is the case with Toda's classical result (see \cite{BE20}*{Proposition~5.1}). The $\R$-motivic analogue can be stated as follows:

\begin{thm}[$\R$-motivic unique realization theorem]  \label{TUT} Let $\rX \in \Sp^\R_{2, \mr{fin}} $ such that 
\[ \Ext^{-1, f, 0}_{\cAR}(\rHR^{*,*}(\rX), \rHR^{*,*}(\rX) ) = 0\]
for any $f \geq 2$. Then any spectrum $\rX'  \in \SpfinR$  for which there exists an $\cAR$-module isomorphism $\rHR^{*,*}(\rX' ) \cong \rHR^{*,*}(\rX )$, is weakly equivalent to $\rX$. {In other words, the $\cAR$-module $\rHR^{*,*}(\rX )$ is uniquely realized in $\SpfinR$ up to a weak equivalence. }
\end{thm}

\begin{proof}The result follows from the fact that the nonzero element 
\[ \iota \in \Ext^{0, 0, 0}_{\cAR}(\rHR^{*,*}(\rX'), \rHR^{*,*}(\rX) ) \]
representing the isomorphism $\rHR^{*,*}(\rX') \cong \rHR^{*,*}(\rX) $ survives the Adams spectral sequence converging to $[\rX, \rX']$. 
\end{proof}

The uniqueness theorem applies to the $\cAR$-modules $\mathcal{B}_{{\sf h}}^\R(1)$ and $\mathcal{B}_{2}^\R(1)$ (see \autoref{uniqueB}). However,  it does not apply to $\cAR_{1}[\overline{\mr{v}}]$ for any $\overline{\mr{v}} \in \mathcal{V}$. Potentially, there can be multiple different finite spectra realizing  $\cAR_{ \overline{\mr{v}}}(1)$ up to a weak equivalence (see \autoref{nonuniqueA}), 
 making it difficult to  get a precise count of the number of $1$-periodic $v_1$-self-maps on $\Y_{({\sf h}, 1)}^\R$ and $\Y_{(2, 1)}^\R$ from \autoref{thm:exactcofiber}.

 Upon applying the Betti realization functor 
\[ 
\begin{tikzcd}
\upbeta: \Sp^{\R} \rar & \Sp^{\Ctwo}
\end{tikzcd} \]
we get various $\Ctwo$-equivariant maps $\upbeta(v): \Sigma^{2,1} \Y^\Ctwo_{(\epsilon, 1)} \longrightarrow \Y^\Ctwo_{(\delta, 1)}$ (where $\epsilon, \delta \in \{ 2, \sfh \}$) whose underlying maps are $v_1$-self-maps of $\Y$. We also get  the following corollary of \autoref{thm:128R} (see \autoref{nochange}). 

\begin{cor} \label{cor:128C} There exists $128$ different $\cA^\Ctwo$-modules whose underlying $\cA^{\Ctwo}(1)$-module structures are free on one generator, all of which can be realized as the $\RO(\Ctwo)$-graded cohomology of a $2$-local finite $\Ctwo$-spectrum. 
\end{cor}

\begin{notn} Let $\cA_1^{\Ctwo}[\overline{\mr{v}}]$ denote the Betti realization $\upbeta(\cA_1^{\R}[\overline{\mr{v}}]) $, where $\overline{\mr{v}} \in \mathcal{V}$.
\end{notn}

Let $\Res: \Sp^\Ctwo \to \Sp$ denote the restriction functor (restricting the $\Ctwo$-action to the trivial group) and  $ \Phi: \Sp^\Ctwo \to \Sp$ denote the geometric fixed-point functor. Note that the geometric fixed-point spectra
\[ 
\GFix{\Y^{\Ctwo}_{({\sfh}, 1)}} := \GFix{\mr{C}^{\Ctwo}(\sfh) \sma \mr{C}^{\Ctwo}(\eta_{1,1})}  \simeq \mr{M}_2(1) \vee \Sigma \mr{M}_2(1)
\]
and
\[ 
\GFix{\Y^{\Ctwo}_{(2, 1)} } := \GFix{\mr{C}^{\Ctwo}(2) \sma \mr{C}^{\Ctwo}(\eta_{1,1})}  \simeq \mr{M}_2(1) \sma  \mr{M}_2(1)
\]
are both of type $1$ (see \cite{BGL}). Further, for degree reasons  
\[ 
\begin{tikzcd}
\GFix{\upbeta(v)}:  \Sigma\GFix{\Y^{\Ctwo}_{(\epsilon, 1)} } \rar &  \GFix{\Y^{\Ctwo}_{(\updelta, 1)} }
\end{tikzcd}
\]
cannot be a $v_1$-self-map, and hence must be nilpotent, using \cite{Thick}. Therefore, {the fiber} $\GFix{\cA_1^\Ctwo[\overline{\mr{v}}]}$ is a type $1$ spectrum, i.e.  $\cA_1^\Ctwo[\overline{\mr{v}}]$ is of type $(2,1)$ in the sense of \cite{BGL}. Much more can be said about $\GFix{\cA_1^\Ctwo[\overline{\mr{v}}]}$ than just its type. In this paper, we give a complete description of the $\cA$-module structure of $\rH^*(\GFix{\cA_1^\Ctwo[\overline{\mr{v}}]})$ for all $\overline{\mr{v}} \in \mathcal{V}$ by developing a general method that compares the $\RO(\Ctwo)$-graded squaring operations of a $\Ctwo$ spectrum with the ordinary squaring operations of its underlying spectrum  as well as its geometric fixed-point spectrum (compare \cite{BW}*{$\mathsection$3}). 

Since $\ParRes{\eqHF} \simeq \rH\F$ we have a natural map 
\[ 
\begin{tikzcd}
\Res_*: \rHC^{n, m}(\mr{E}) \simeq [\mr{E}, \Sigma^{n,m} \rH\uF ] \rar &   \text{[} \ParRes{\mr{E}}, \Sigma^{n}  \ParRes{ \rH\uF} \text{]} \simeq \rH^{n}(\ParRes{\mr{E}}) 
\end{tikzcd} 
\] 
for any $\mr{E} \in \Sp^\Ctwo$. We  use the following theorem to identify the spectrum underlying $\cA_1^\Ctwo[\overline{\mr{v}}]$ (see \autoref{underA1}). 

\begin{thm} \label{thm:squareunder} 
For $\mr{E}\in \SpfinC$ and 
any class $u \in \rHC^{\star}\!(\mr{E})$, 
$\Res_*(\eSq^n(u)) = \Sq^n(\Res_*(u))$.
\end{thm}

Using the fact that the projection map 
\[ \begin{tikzcd}
\uppi_{\F}^{(0)}: \GFix{\rH\uF} \rar[two heads] & \rH\F,
\end{tikzcd} \]
 is an $\mathbb{E}_{\infty}$-ring map, one defines (also see \cite{BW}*{(2.7)})  the map 
\begin{equation}
\begin{tikzcd}
\modPhi\colon \rHC^{n,m}(\mr{E}) \rar[two heads]  &  \rH^{n-m}(\GFix{\mr{E}}) 
\end{tikzcd}
\end{equation}
which compares the $\RO(\Ctwo)$-graded cohomology of a $\Ctwo$-spectrum $\mr{E}$ with the ordinary cohomology of its geometric fixed-point spectrum. We show: 

 \begin{thm} \label{thm:squarefix} For $\mr{E}\in \SpfinC$ and any class $u \in \rHC^{\star}\!(\mr{E})$, 
\[\modPhi(\eSq^{2n}(u)) = \Sq^n(\modPhi(u)).\] 
\end{thm}

\begin{figure}[h]
\caption{Some underlying and fixed $\cA$-modules of $\cAC_1$}
\label{UnderFixAmodule}
\begin{center}
\begin{tabular}{|c|c|c|c||}
\toprule
$\overline{\mr{v}} \in \mathcal{V}$  &   $\rH^*(\ParRes{ \cA_1^\Ctwo[\overline{\mr{v}}]})$    &   $\rH^*(\GFix{ \cA_1^\Ctwo[\overline{\mr{v}}]})$  & Cofiber of \\
\midrule
$(0,0,1,0,0,0,0)$ &
\raisebox{-3em}{$
\begin{tikzpicture}\begin{scope}[thick, every node/.style={sloped,allow upside down}, scale=0.32]
\draw (0,0)  node[inner sep=0] (v00) {} -- (0,1) node[inner sep=0] (v01) {};
\draw (0,2)  node[inner sep=0] (v02) {} -- (0,3) node[inner sep=0] (v03) {};
\draw (1,3)  node[inner sep=0] (v13) {} -- (1,4) node[inner sep=0] (v14) {};
\draw (1,5)  node[inner sep=0] (v15) {} -- (1,6) node[inner sep=0] (v16) {};
 \draw [color=blue] (v00) to [out=150,in=-150] (v02);
 \draw [color=blue] (v01) to [out=15,in=-90] (v13);
 \draw [color=blue] (v03) to [out=90,in=-165] (v15);
 \draw [color=blue] (v14) to [out=30,in=-30] (v16);
 \draw [color=blue] (v02) to [out=15,in=-150] (v14);
 \draw [color = red] (v01) to (2,1) to (2,5) to (v15);
\filldraw (v00) circle (2.5pt);
\filldraw (v01) circle (2.5pt);
\filldraw (v02) circle (2.5pt);
\filldraw (v03) circle (2.5pt);
\filldraw (v13) circle (2.5pt);
\filldraw (v14) circle (2.5pt);
\filldraw (v15) circle (2.5pt);
\filldraw (v16) circle (2.5pt);
\end{scope}\end{tikzpicture}
$}
&
\raisebox{-3em}{$
\begin{tikzpicture}\begin{scope}[ thick, every node/.style={sloped,allow upside down}, scale=0.32]
\draw (1,0)  node[inner sep=0] (v10) {} -- (1,1) node[inner sep=0] (v01) {};
\draw (2,1)  node[inner sep=0] (v21) {} -- (2,2) node[inner sep=0] (v12) {};
\draw (4,2)  node[inner sep=0] (v42) {} -- (5,3) node[inner sep=0] (v53) {};
\draw (3,3)  node[inner sep=0] (v33) {} -- (4,4) node[inner sep=0] (v44) {};
 \draw [color=blue] (v21) to [out=15,in=-150] (v33);
\filldraw (v10) circle (2.5pt);
\filldraw (v21) circle (2.5pt);
\filldraw (v01) circle (2.5pt);
\filldraw (v12) circle (2.5pt);
\filldraw (v42) circle (2.5pt);
\filldraw (v53) circle (2.5pt);
\filldraw (v33) circle (2.5pt);
\filldraw (v44) circle (2.5pt);
\draw (v42) to (v33);
\draw (v53) to (v44);
\end{scope}\end{tikzpicture}
$} 
& 
$v: \Sigma^{2,1}\Y_{(\sfh,1)} \to \Y_{(\sfh,1)}$
 \\
 \midrule
 $(1,1,0,0,0,0,1)$ &
\raisebox{-3em}{$
\begin{tikzpicture}\begin{scope}[thick, every node/.style={sloped,allow upside down}, scale=0.32]
\draw (0,0)  node[inner sep=0] (v00) {} -- (0,1) node[inner sep=0] (v01) {};
\draw (0,2)  node[inner sep=0] (v02) {} -- (0,3) node[inner sep=0] (v03) {};
\draw (1,3)  node[inner sep=0] (v13) {} -- (1,4) node[inner sep=0] (v14) {};
\draw (1,5)  node[inner sep=0] (v15) {} -- (1,6) node[inner sep=0] (v16) {};
 \draw [color=blue] (v00) to [out=150,in=-150] (v02);
 \draw [color=blue] (v01) to [out=15,in=-90] (v13);
 \draw [color=blue] (v03) to [out=90,in=-165] (v15);
 \draw [color=blue] (v14) to [out=30,in=-30] (v16);
 \draw [color=blue] (v02) to [out=15,in=-150] (v14);
 \draw [color = red] (v01) to (2,1) to (2,5) to (v15);
\filldraw (v00) circle (2.5pt);
\filldraw (v01) circle (2.5pt);
\filldraw (v02) circle (2.5pt);
\filldraw (v03) circle (2.5pt);
\filldraw (v13) circle (2.5pt);
\filldraw (v14) circle (2.5pt);
\filldraw (v15) circle (2.5pt);
\filldraw (v16) circle (2.5pt);
\end{scope}\end{tikzpicture}
$}
&
\raisebox{-3em}{$
\begin{tikzpicture}\begin{scope}[ thick, every node/.style={sloped,allow upside down}, scale=0.32]
\draw (3,0)  node[inner sep=0] (v30) {} -- (3,1) node[inner sep=0] (v31) {};
\draw (1,1)  node[inner sep=0] (v11) {} -- (2,2) node[inner sep=0] (v22) {};
\draw (0,2)  node[inner sep=0] (v02) {} -- (1,3) node[inner sep=0] (v13) {};
\draw (3,3)  node[inner sep=0] (v33) {} -- (3,4) node[inner sep=0] (v34) {};
 \draw [color=blue] (v31) to [bend right=50] (v33);
 \draw [color=blue] (v30) to [bend left] (v22);
 \draw [color=blue] (v22) to [bend left] (v34);
\filldraw (v30) circle (2.5pt);
\filldraw (v31) circle (2.5pt);
\filldraw (v11) circle (2.5pt);
\filldraw (v02) circle (2.5pt);
\filldraw (v22) circle (2.5pt);
\filldraw (v13) circle (2.5pt);
\filldraw (v33) circle (2.5pt);
\filldraw (v34) circle (2.5pt);
\draw (v11) to (v02);
\draw (v22) to (v13);
\draw (0,0) node[left]{$ $} (0,1) node[left]{$ $} (1,1) node[right]{$ $} (1,2) node[left]{$ $};
\draw (3,4) node[right]{$ $} (3,3) node[right]{$ $} (2,3) node[left]{$ $} (2,2) node[right]{$ $};
\end{scope}\end{tikzpicture}
$}
& 
$v: \Sigma^{2,1}\Y_{(2,1)} \to \Y_{(\sfh,1)}$
 \\
 \midrule 
 $(0,1,0,1,0,1,0)$ &
\raisebox{-3em}{$
\begin{tikzpicture}\begin{scope}[thick, every node/.style={sloped,allow upside down}, scale=0.32]
\draw (0,0)  node[inner sep=0] (v00) {} -- (0,1) node[inner sep=0] (v01) {};
\draw (0,2)  node[inner sep=0] (v02) {} -- (0,3) node[inner sep=0] (v03) {};
\draw (1,3)  node[inner sep=0] (v13) {} -- (1,4) node[inner sep=0] (v14) {};
\draw (1,5)  node[inner sep=0] (v15) {} -- (1,6) node[inner sep=0] (v16) {};
 \draw [color=blue] (v00) to [out=150,in=-150] (v02);
 \draw [color=blue] (v01) to [out=15,in=-90] (v13);
 \draw [color=blue] (v03) to [out=90,in=-165] (v15);
 \draw [color=blue] (v14) to [out=30,in=-30] (v16);
 \draw [color=blue] (v02) to [out=15,in=-150] (v14);
 \draw [color=red] (v02) to (-1,2) to (-1,6) to (v16);
 \draw [color = red] (v01) to (2,1) to (2,5) to (v15);
\filldraw (v00) circle (2.5pt);
\filldraw (v01) circle (2.5pt);
\filldraw (v02) circle (2.5pt);
\filldraw (v03) circle (2.5pt);
\filldraw (v13) circle (2.5pt);
\filldraw (v14) circle (2.5pt);
\filldraw (v15) circle (2.5pt);
\filldraw (v16) circle (2.5pt);
\end{scope}\end{tikzpicture}
$}
&
\raisebox{-3em}{$
\begin{tikzpicture}\begin{scope}[thick, every node/.style={sloped,allow upside down}, scale=0.32]
\draw (-.5,0)  node[inner sep=0] (v00) {} -- (-.5,1) node[inner sep=0] (v01) {};
\draw (.5,1)  node[inner sep=0] (v11) {} -- (.5,2) node[inner sep=0] (v02) {};
\draw (1.5,2)  node[inner sep=0] (v12) {} -- (1.5,3) node[inner sep=0] (v13) {};
\draw (2.5,3)  node[inner sep=0] (v23) {} -- (2.5,4) node[inner sep=0] (v24) {};
 \draw [color=blue] (v00) to [out=15,in=-150] (v02);
 \draw [color=blue] (v11) to [out=15,in=-150] (v13);
 \draw [color=blue] (v12) to [out=30,in=-500] (v24);
 \draw [color = red] (v00) to (3,0) to (3,4) to (v24);
\filldraw (v00) circle (2.5pt);
\filldraw (v01) circle (2.5pt);
\filldraw (v11) circle (2.5pt);
\filldraw (v02) circle (2.5pt);
\filldraw (v12) circle (2.5pt);
\filldraw (v23) circle (2.5pt);
\filldraw (v13) circle (2.5pt);
\filldraw (v24) circle (2.5pt);
\end{scope}\end{tikzpicture}
$}
& 
$v: \Sigma^{2,1}\Y_{(2,1)} \to \Y_{(2,1)}$
 \\
 \midrule
 $(1,0,0,0,0,1,1)$ &
\raisebox{-3em}{$
\begin{tikzpicture}\begin{scope}[thick, every node/.style={sloped,allow upside down}, scale=0.32]
\draw (0,0)  node[inner sep=0] (v00) {} -- (0,1) node[inner sep=0] (v01) {};
\draw (0,2)  node[inner sep=0] (v02) {} -- (0,3) node[inner sep=0] (v03) {};
\draw (1,3)  node[inner sep=0] (v13) {} -- (1,4) node[inner sep=0] (v14) {};
\draw (1,5)  node[inner sep=0] (v15) {} -- (1,6) node[inner sep=0] (v16) {};
 \draw [color=blue] (v00) to [out=150,in=-150] (v02);
 \draw [color=blue] (v01) to [out=15,in=-90] (v13);
 \draw [color=blue] (v03) to [out=90,in=-165] (v15);
 \draw [color=blue] (v14) to [out=30,in=-30] (v16);
 \draw [color=blue] (v02) to [out=15,in=-150] (v14);
 \draw [color=red] (v02) to (-0.75,2) to (-0.75,6) to (v16);
 \draw [color = red] (v01) to (2,1) to (2,5) to (v15);
\draw [color = red] (v00) to (-1.5,0) to (-1.5,4) to (v14);
\filldraw (v00) circle (2.5pt);
\filldraw (v01) circle (2.5pt);
\filldraw (v02) circle (2.5pt);
\filldraw (v03) circle (2.5pt);
\filldraw (v13) circle (2.5pt);
\filldraw (v14) circle (2.5pt);
\filldraw (v15) circle (2.5pt);
\filldraw (v16) circle (2.5pt);
\end{scope}\end{tikzpicture}
$}
&
\raisebox{-3em}{$
\begin{tikzpicture}\begin{scope}[ thick, every node/.style={sloped,allow upside down}, scale=0.32]
\draw (3,0)  node[inner sep=0] (v30) {} -- (3,1) node[inner sep=0] (v31) {};
\draw (1,1)  node[inner sep=0] (v11) {} -- (2,2) node[inner sep=0] (v22) {};
\draw (0,2)  node[inner sep=0] (v02) {} -- (1,3) node[inner sep=0] (v13) {};
\draw (3,3)  node[inner sep=0] (v33) {} -- (3,4) node[inner sep=0] (v34) {};
 \draw [color=blue] (v31) to [bend right=50] (v33);
 \draw [color=blue] (v30) to [bend left] (v22);
 \draw [color=blue] (v22) to [bend left] (v34);
 \draw [color=blue] (v11) to [out=75,in=-105] (v13);
\filldraw (v30) circle (2.5pt);
\filldraw (v31) circle (2.5pt);
\filldraw (v11) circle (2.5pt);
\filldraw (v02) circle (2.5pt);
\filldraw (v22) circle (2.5pt);
\filldraw (v13) circle (2.5pt);
\filldraw (v33) circle (2.5pt);
\filldraw (v34) circle (2.5pt);
\draw (v11) to (v02);
\draw (v22) to (v13);
\draw (0,0) node[left]{$ $} (0,1) node[left]{$ $} (1,1) node[right]{$ $} (1,2) node[left]{$ $};
\draw (3,4) node[right]{$ $} (3,3) node[right]{$ $} (2,3) node[left]{$ $} (2,2) node[right]{$ $};
\end{scope}\end{tikzpicture}
$}
& 
$v: \Sigma^{2,1}\Y_{(\sfh,1)} \to \Y_{(2,1)}$
 \\
 \midrule 
 $(0,0,0,1,0,1,0)$ &
\raisebox{-3em}{$
\begin{tikzpicture}\begin{scope}[thick, every node/.style={sloped,allow upside down}, scale=0.32]
\draw (0,0)  node[inner sep=0] (v00) {} -- (0,1) node[inner sep=0] (v01) {};
\draw (0,2)  node[inner sep=0] (v02) {} -- (0,3) node[inner sep=0] (v03) {};
\draw (1,3)  node[inner sep=0] (v13) {} -- (1,4) node[inner sep=0] (v14) {};
\draw (1,5)  node[inner sep=0] (v15) {} -- (1,6) node[inner sep=0] (v16) {};
 \draw [color=blue] (v00) to [out=150,in=-150] (v02);
 \draw [color=blue] (v01) to [out=15,in=-90] (v13);
 \draw [color=blue] (v03) to [out=90,in=-165] (v15);
 \draw [color=blue] (v14) to [out=30,in=-30] (v16);
 \draw [color=blue] (v02) to [out=15,in=-150] (v14);
 \draw [color=red] (v02) to (-0.75,2) to (-0.75,6) to (v16);
 \draw [color = red] (v01) to (2,1) to (2,5) to (v15);
\draw [color = red] (v00) to (-1.5,0) to (-1.5,4) to (v14);
\filldraw (v00) circle (2.5pt);
\filldraw (v01) circle (2.5pt);
\filldraw (v02) circle (2.5pt);
\filldraw (v03) circle (2.5pt);
\filldraw (v13) circle (2.5pt);
\filldraw (v14) circle (2.5pt);
\filldraw (v15) circle (2.5pt);
\filldraw (v16) circle (2.5pt);
\end{scope}\end{tikzpicture}
$}
&
\raisebox{-3em}{$
\begin{tikzpicture}\begin{scope}[ thick, every node/.style={sloped,allow upside down}, scale=0.32]
\draw (0,0)  node[inner sep=0] (v00) {} -- (0,1) node[inner sep=0] (v01) {};
\draw (1,1)  node[inner sep=0] (v11) {} -- (1,2) node[inner sep=0] (v12) {};
\draw (2,2)  node[inner sep=0] (v22) {} -- (2,3) node[inner sep=0] (v23) {};
\draw (3,3)  node[inner sep=0] (v33) {} -- (3,4) node[inner sep=0] (v34) {};
 \draw [color=blue] (v11) to [out=15,in=-150] (v23);
  \draw [color=blue] (v22) to [out=15,in=210] (v34);
  \draw [red] (v00) to (3.5,0) to (3.5,4) to (v34);
\filldraw (v00) circle (2.5pt);
\filldraw (v11) circle (2.5pt);
\filldraw (v01) circle (2.5pt);
\filldraw (v12) circle (2.5pt);
\filldraw (v22) circle (2.5pt);
\filldraw (v23) circle (2.5pt);
\filldraw (v33) circle (2.5pt);
\filldraw (v34) circle (2.5pt);
\end{scope}\end{tikzpicture}
$} 
& 
$v: \Sigma^{2,1}\Y_{(2,1)} \to \Y_{(\sfh,1)}$
 \\
 \midrule
 $(1,0,0,0,0,0,0)$ &
\raisebox{-3em}{$
\begin{tikzpicture}\begin{scope}[thick, every node/.style={sloped,allow upside down}, scale=0.32]
\draw (0,0)  node[inner sep=0] (v00) {} -- (0,1) node[inner sep=0] (v01) {};
\draw (0,2)  node[inner sep=0] (v02) {} -- (0,3) node[inner sep=0] (v03) {};
\draw (1,3)  node[inner sep=0] (v13) {} -- (1,4) node[inner sep=0] (v14) {};
\draw (1,5)  node[inner sep=0] (v15) {} -- (1,6) node[inner sep=0] (v16) {};
 \draw [color=blue] (v00) to [out=150,in=-150] (v02);
 \draw [color=blue] (v01) to [out=15,in=-90] (v13);
 \draw [color=blue] (v03) to [out=90,in=-165] (v15);
 \draw [color=blue] (v14) to [out=30,in=-30] (v16);
 \draw [color=blue] (v02) to [out=15,in=-150] (v14);
\draw [color = red] (v00) to (-1.5,0) to (-1.5,4) to (v14);
 \draw [color = red] (v01) to (2,1) to (2,5) to (v15);
\filldraw (v00) circle (2.5pt);
\filldraw (v01) circle (2.5pt);
\filldraw (v02) circle (2.5pt);
\filldraw (v03) circle (2.5pt);
\filldraw (v13) circle (2.5pt);
\filldraw (v14) circle (2.5pt);
\filldraw (v15) circle (2.5pt);
\filldraw (v16) circle (2.5pt);
\end{scope}\end{tikzpicture}
$}
&
\raisebox{-3em}{$
\begin{tikzpicture}\begin{scope}[ thick, every node/.style={sloped,allow upside down}, scale=0.32]
\draw (0,0)  node[inner sep=0] (v00) {} -- (0,1) node[inner sep=0] (031) {};
\draw (2,1)  node[inner sep=0] (v21) {} -- (2,2) node[inner sep=0] (v22) {};
\draw (1,2)  node[inner sep=0] (v12) {} -- (1,3) node[inner sep=0] (v13) {};
\draw (3,3)  node[inner sep=0] (v33) {} -- (3,4) node[inner sep=0] (v34) {};
 \draw [color=blue] (v00) to [out=13, in=210] (v12);
 \draw [color=blue] (v21) to [out=150,in=-30] (v13);
\filldraw (v00) circle (2.5pt);
\filldraw (v01) circle (2.5pt);
\filldraw (v21) circle (2.5pt);
\filldraw (v22) circle (2.5pt);
\filldraw (v12) circle (2.5pt);
\filldraw (v13) circle (2.5pt);
\filldraw (v33) circle (2.5pt);
\filldraw (v34) circle (2.5pt);
\draw (0,0) node[left]{$ $} (0,1) node[left]{$ $} (1,1) node[right]{$ $} (1,2) node[left]{$ $};
\draw (3,4) node[right]{$ $} (3,3) node[right]{$ $} (2,3) node[left]{$ $} (2,2) node[right]{$ $};
\end{scope}\end{tikzpicture}
$}
& 
$v: \Sigma^{2,1}\Y_{(\sfh,1)} \to \Y_{(2,1)}$
 \\
 \midrule 
 $(1,0,0,0,0,0,1)$ &
\raisebox{-3em}{$
\begin{tikzpicture}\begin{scope}[thick, every node/.style={sloped,allow upside down}, scale=0.32]
\draw (0,0)  node[inner sep=0] (v00) {} -- (0,1) node[inner sep=0] (v01) {};
\draw (0,2)  node[inner sep=0] (v02) {} -- (0,3) node[inner sep=0] (v03) {};
\draw (1,3)  node[inner sep=0] (v13) {} -- (1,4) node[inner sep=0] (v14) {};
\draw (1,5)  node[inner sep=0] (v15) {} -- (1,6) node[inner sep=0] (v16) {};
 \draw [color=blue] (v00) to [out=150,in=-150] (v02);
 \draw [color=blue] (v01) to [out=15,in=-90] (v13);
 \draw [color=blue] (v03) to [out=90,in=-165] (v15);
 \draw [color=blue] (v14) to [out=30,in=-30] (v16);
 \draw [color=blue] (v02) to [out=15,in=-150] (v14);
 \draw [color = red] (v01) to (2,1) to (2,5) to (v15);
  \draw [color = red] (v00) to (-1,0) to (-1,4) to (v14);
\filldraw (v00) circle (2.5pt);
\filldraw (v01) circle (2.5pt);
\filldraw (v02) circle (2.5pt);
\filldraw (v03) circle (2.5pt);
\filldraw (v13) circle (2.5pt);
\filldraw (v14) circle (2.5pt);
\filldraw (v15) circle (2.5pt);
\filldraw (v16) circle (2.5pt);
\end{scope}\end{tikzpicture}
$}
&
\raisebox{-3em}{$
\begin{tikzpicture}\begin{scope}[ thick, every node/.style={sloped,allow upside down}, scale=0.3]
\draw (0,0)  node[inner sep=0] (v00) {} -- (0,1) node[inner sep=0] (v11) {};
\draw (2,1)  node[inner sep=0] (v21) {} -- (2,2) node[inner sep=0] (v22) {};
\draw (1,2)  node[inner sep=0] (v12) {} -- (1,3) node[inner sep=0] (v13) {};
\draw (3,3)  node[inner sep=0] (v33) {} -- (3,4) node[inner sep=0] (v34) {};
 \draw [color=blue] (v00) to [out=15,in=-150] (v12);
 \draw [color=blue] (v21) to [out=-150,in=15] (v13);
\draw [color=blue] (v22) to [out=15,in=-165] (v34);
\filldraw (v00) circle (2.5pt);
\filldraw (v11) circle (2.5pt);
\filldraw (v21) circle (2.5pt);
\filldraw (v22) circle (2.5pt);
\filldraw (v12) circle (2.5pt);
\filldraw (v13) circle (2.5pt);
\filldraw (v33) circle (2.5pt);
\filldraw (v34) circle (2.5pt);
\end{scope}\end{tikzpicture}
$} 
& 
$v: \Sigma^{2,1}\Y_{(2,1)} \to \Y_{(2,1)}$
 \\
 \midrule
 $(1,1,1,1,1,0,1)$ &
\raisebox{-3em}{$
\begin{tikzpicture}\begin{scope}[thick, every node/.style={sloped,allow upside down}, scale=0.32]
\draw (0,0)  node[inner sep=0] (v00) {} -- (0,1) node[inner sep=0] (v01) {};
\draw (0,2)  node[inner sep=0] (v02) {} -- (0,3) node[inner sep=0] (v03) {};
\draw (1,3)  node[inner sep=0] (v13) {} -- (1,4) node[inner sep=0] (v14) {};
\draw (1,5)  node[inner sep=0] (v15) {} -- (1,6) node[inner sep=0] (v16) {};
 \draw [color=blue] (v00) to [out=150,in=-150] (v02);
 \draw [color=blue] (v01) to [out=15,in=-90] (v13);
 \draw [color=blue] (v03) to [out=90,in=-165] (v15);
 \draw [color=blue] (v14) to [out=30,in=-30] (v16);
 \draw [color=blue] (v02) to [out=15,in=-150] (v14);
 \draw [color = red] (v01) to (2,1) to (2,5) to (v15);
\draw [color = red] (v00) to (-1.5,0) to (-1.5,4) to (v14);
\filldraw (v00) circle (2.5pt);
\filldraw (v01) circle (2.5pt);
\filldraw (v02) circle (2.5pt);
\filldraw (v03) circle (2.5pt);
\filldraw (v13) circle (2.5pt);
\filldraw (v14) circle (2.5pt);
\filldraw (v15) circle (2.5pt);
\filldraw (v16) circle (2.5pt);
\end{scope}\end{tikzpicture}
$}
&
\raisebox{-3em}{$
\begin{tikzpicture}\begin{scope}[ thick, every node/.style={sloped,allow upside down}, scale=0.32]
\draw (0,0)  node[inner sep=0] (v00) {} -- (0,1) node[inner sep=0] (v01) {};
\draw (0,2)  node[inner sep=0] (v02) {} -- (0,3) node[inner sep=0] (v03) {};
\draw (2,1)  node[inner sep=0] (v11) {} -- (2,2) node[inner sep=0] (v12) {};
\draw (2,3)  node[inner sep=0] (v13) {} -- (2,4) node[inner sep=0] (v14) {};
 \draw [color=blue] (v00) to [out=30,in=-30] (v02);
  \draw [color=blue] (v01) to [out=150,in=-150] (v03);
  \draw [color=blue] (v11) to [out=30,in=-30] (v13);
   \draw [color=blue] (v12) to [out=150,in=-150] (v14);
\draw [color = red] (v00) to (1,0) to (1,4) to (v14);
\filldraw (v00) circle (2.5pt);
\filldraw (v01) circle (2.5pt);
\filldraw (v02) circle (2.5pt);
\filldraw (v03) circle (2.5pt);
\filldraw (v11) circle (2.5pt);
\filldraw (v12) circle (2.5pt);
\filldraw (v13) circle (2.5pt);
\filldraw (v14) circle (2.5pt);
\draw (0,0) node[left]{$ $} (0,1) node[left]{$ $} (1,1) node[right]{$ $} (1,2) node[left]{$ $};
\draw (3,4) node[right]{$ $} (3,3) node[right]{$ $} (2,3) node[left]{$ $} (2,2) node[right]{$ $};
\end{scope}\end{tikzpicture}
$}
& 
$v: \Sigma^{2,1}\Y_{(2,1)} \to \Y_{(2,1)}$
 \\
 \midrule 
\end{tabular}
\end{center}
\end{figure}

We find  \autoref{thm:squareunder}  and \autoref{thm:squarefix} very handy for computational purposes. These results can be applied to understand the  $\RO(\Ctwo)$-graded squaring operations on the cohomology of a wide variety of $\Ctwo$-spectra whose underlying and geometric fixed-point spectra are known. Alternatively, one can identify the action of the classical Steenrod algebra on the cohomology of the underlying as well as the geometric fixed-points of a $\Ctwo$-spectrum from the knowledge of $\RO(\Ctwo)$-graded Steenrod operations. We apply \autoref{thm:squareunder} and \autoref{thm:squarefix} to identify the $\cA$-module structure of the underlying and the geometric fixed-points of $\cA_{1}^{\Ctwo}[ \overline{\mr{v}}]$ (see \autoref{underA1} and \autoref{GFixA1}).

In \autoref{UnderFixAmodule}, we provide the $\cA$-module structure of the underlying and the geometric fixed points of {$\mr{A}_1^{\Ctwo}[\overline{\mr{v}}]$} for selected values of $\overline{\mr{v}} \in \mathcal{V}$. We express $\Sq^1$, $\Sq^2$ and $\Sq^4$ using black, blue, and red lines respectively.

\begin{rmk}[Appearance of the Joker] 
 We note that the  $\cA(1)$-module 
 \[ 
 \begin{tikzpicture}
 \begin{scope}[ thick, every node/.style={sloped,allow upside down}, scale=0.32]
\draw (3,0)  node[inner sep=0] (v30) {} -- (3,1) node[inner sep=0] (v31) {};
\draw  (2.5,2) node[inner sep=0] (v22) {};
\draw (3,3)  node[inner sep=0] (v33) {} -- (3,4) node[inner sep=0] (v34) {};
 \draw [color=blue] (v31) to [bend right=70] (v33);
 \draw [color=blue] (v30) to [bend left = 50] (v22);
 \draw [color=blue] (v22) to [bend left =50] (v34);
\filldraw (v30) circle (2.5pt);
\filldraw (v31) circle (2.5pt);
\filldraw (v22) circle (2.5pt);
\filldraw (v33) circle (2.5pt);
\filldraw (v34) circle (2.5pt);
\draw (0,0) node[left]{$ $} (0,1) node[left]{$ $} (1,1) node[right]{$ $} (1,2) node[left]{$ $};
\draw (3,4) node[right]{$ $} (3,3) node[right]{$ $} (2,3) node[left]{$ $} (2,2) node[right]{$ $};
\end{scope}\end{tikzpicture},
 \]
  often called the Joker,  is a subcomplex of the geometric fixed point of $\mr{A}_1^\R[\overline{\mr{v}}]$ if and only if $\jay_{24} =1$. Further, when $\jay_{24} =1$ then in \eqref{exactBAB},  $\epsilon$ and $\delta$ cannot both
  equal $\sfh$. This  can  easily be derived from \autoref{list}  and \autoref{thm:squarefix}. 
\end{rmk}

\begin{rmk} In \cite{BGL}, the authors construct $\mr{A}_1^\R[{\bf 0}]$ as a split summand of $\mathcal{Q}_\R^{\sma 3}$ using a certain idempotent of $\Z_{(2)}[\Sigma_3]$. Let $\tilde{\mathcal{Q}}_\R \in \SpfinR$ be such that its cohomology as an $\cA^\R(1)$-module is  isomorphic to $\mathcal{Q}_\R$, but has the additional relation \[ \Sq^4(a) = \rho \cdot c \] (in the notation of \cite[Figure~3.6]{BGL}), as an $\cA^\R$-module. If we replace $\mathcal{Q}_\R$ by a complex  $\tilde{\mathcal{Q}}_\R$ in \cite{BGL},   we get $\mr{A}_1^{\Ctwo}[\overline{\mr{v}}]$, where $\overline{\mr{v}} =(1,1,1,1,1,0,1)$ (see the last diagram in \autoref{UnderFixAmodule}). 
\end{rmk}

\begin{rmk}
The classical spectrum $\cA_1$ is a type $2$ spectrum and supports a $32$ periodic $v_2$-self-map \cite{BEM}. It remains to be seen if this $v_2$-self-map can be lifted to $\cA_{1}^\R[\overline{\mr{v}}]$ for various $\overline{\mr{v}} \in \mathcal{V}$.
\end{rmk}

 Recently in \cite{BE20}, the authors introduced a new type $2$ spectrum $\mathcal{Z}$ which is notable for admitting a $v_2$-self-map of lowest possible periodicity, that is $1$.  The low periodicity of the $v_2$-self-map makes the spectrum $\mathcal{Z}$ suitable for the analysis of the telescope conjecture which, if true, would imply that the natural map from the telescope of $\mathcal{Z}$ to the $\mr{K}(2)$-localization of $\mathcal{Z}$ is a weak equivalence.  While the telescope conjecture is true  for finite spectra of type $1$ \cites{M81,M82, MillerTel}, it is expected  to be false for finite spectra of type $\geq 2$ (see \cite{MRS}).  In fact, in \cite{BBBCX}, the authors study the  {prime $2$,  height $2$ } telescope conjecture 
 using the spectrum $\mathcal{Z}$ and lay down several conjectures ({see \cite{BBBCX}*{$\mathsection$9}}),  whose validity would lead to a disproof of the telescope conjecture. In this paper, we also construct an $\R$-motivic analogue of $\mathcal{Z}$ which  is likely to  shed light on some of these conjectures. 
 
 \begin{thm} \label{thm:Z} There exists  $\mathcal{Z}_\R \in \Sp^{\R}_{2, \mr{fin}}$  such that the underlying $\cAR(2)$-module structure of its cohomology is isomorphic to  
  \[ \rHR^{*,*}(\mathcal{Z}_\R) \cong_{\cAR(2)} \cAR(2) \otimes_{\Lambda( \tilde{\mr{Q}}_2^\R)}\bMR \]
 where $\tilde{\mr{Q}}_2^\R := [\Sq^4, \mr{Q}_1^\R]$.  
 \end{thm}
 
 In future work, we intend to study the properties of $\mathcal{Z}_\R$ extensively and hope to prove, among other things,  the following conjecture. 
 \begin{conj} \label{conj:v2Z} The spectrum $\mathcal{Z}_\R$ is of type $(2,2)$ and admits a $v_{(2, \mr{nil})}$-self-map
 \[ \begin{tikzcd}
 v: \Sigma^{6,3} \mathcal{Z}_\R  \rar & \mathcal{Z}_\R
 \end{tikzcd}
 \]
 of periodicity $1$. 
 \end{conj}
 
 \subsection*{Acknowledgements} The authors have benefited from  conversations with  Mike Hill, Nick Kuhn, Piotr Pstragowski, Paul VanKoughnett, and Dylan Wilson. 
{  The first author would  also like to acknowledge his debt to Mark Behrens for his relentless support.}
 
  \subsection*{Organization of the paper}
  In \autoref{Sec:Toda}, we discuss the $\R$-motivic Toda realization \autoref{TRT} and derive various weak forms that are suitable for applications. 
In \autoref{sec:Comparison},  we construct the equivariant Steenrod operations using the equivariant extended power construction
and prove \autoref{thm:squareunder} and \autoref{thm:squarefix}, which
establish comparisons  with the classical Steenrod operations. In \autoref{sec:Trealization}, we apply the discussion in \autoref{Sec:Toda} to obtain the $\R$-motivic topological realizations of $\cA^\R(1)$ and  {analyze the properties of their Betti realizations using results from \autoref{sec:Comparison}.} In \autoref{sec:Z}, we construct the $\R$-motivic spectrum $\mathcal{Z}_\R$ using a method of Smith.  Finally, the short \autoref{AdemSec} lists the Adem relations in the $\R$-motivic Steenrod algebra.
  
\section{$\R$-motivic Toda realization theorem} \label{Sec:Toda}

The classical Toda realization theorem \cite{T71} (see also \cite{BE20}*{Theorem 3.1}), is recast in the modern literature as a special case of Goerss-Hopkins obstruction theory \cite{GH} (when the chosen operad is trivial). This obstruction theory can be generalized to the $\R$-motivic setting \cite{MazelG}, and \autoref{TRT} would then be a special case of such a generalization. 

More recent work of \cite{PV} conceptualizes  Goerss-Hopkins obstruction theory in the general setup of stable $\infty$-categories with $t$-structures. If we set $\mathcal{C} = \SpfinR$, $\mr{A} = \bS_{\RHF}$, and let $\mr{K}$ to be a finite $\cA_*^{\R}$-comodule  in  \cite{PV}*{Corollary~4.10}, then we get a sequence of obstruction classes 
\begin{equation} \label{obclass}
 \uptheta_n \in \Ext_{\cA_*^\R}^{-2, n+2, 0}(\mr{K}, \mr{K})
 \end{equation}
for each $n\geq 0$, the vanishing of which guarantees the existence of an $\bS_{\RHF}$-module whose homology is isomorphic to $\mr{K}$ as an $\cA_*^\R$-comodule. Since the $t$-structure in $\Sp^\R$ does not change the motivic weight, the obstruction classes in \eqref{obclass}  lie in the Ext-groups of motivic weight $0$. 

If $\mr{M}$ is a finite $\bM^\R$-free $\cAR$-module then $\mr{K}:= \hom_{\bM^\R}(\mr{M}, \bM^\R)$ is a finite $\cA_*^\R$-comodule,
\[ 
\Ext_{\cA_*^\R}^{*,*,*}(\mr{K}, \mr{K}) \cong \Ext_{\cAR}^{*,*,*}(\mr{M}, \mr{M}),
\]
and therefore, \autoref{TRT} follows. Alternatively, one can prove \autoref{TRT} simply by emulating the classical proof (as exposed in \cite{BE20}*{$\mathsection 3$}).  

The purpose of this section is to prove various weaker forms of the $\R$-motivic Toda realization theorem (\autoref{TRT}), which are perhaps more convenient for  application purposes.  Explicit calculation of  $\Ext_{\cAR}^{*,*,*}(\mr{M}, \mr{M})$ can often be difficult, and one can use  a sequence of spectral sequences to approximate these ext groups. Each such approximation leads to a  corresponding weaker form.

\subsection{Weak $\R$-motivic Toda realization -- version (I)}  \label{subsec:wTRT1}

Let $\mr{M}$ be an $\cAR$-module whose underlying $\bM^\R$-module is free and finitely generated. Let  $\mathcal{B}_{\mr{M}} $ denote its   $\bM^\R$-basis and $\mathcal{D}_{\mr{M}}$ denote the collection of bidegrees  in which there is an  element in $\mathcal{B}_{\mr{M}}$. For any element $x \in \mr{M}^{s, w}$, we  let ${\sf t}(x) = s + w$ and  define
 \[ \mr{M}_{\geq n} := \bM^\R \cdot \{  b \in \mathcal{B}_{\mr{M}}: {\sf t}(b) \geq n  \}  \]
 as the free  sub $\bM^\R$-module of $\mr{M}$ generated by $\{  b \in \mathcal{B}_{\mr{M}}: {\sf t}(b) \geq n  \}$.
 
Note that the $\cAR$-module structure of $\mr{M}$ is determined by the action of $\cAR$ on the elements of $\mathcal{B}_{\mr{M}}$ and the Cartan formula. This, along with the fact that  ${ \sf t}(a) \geq 0 $ for all $a \in \cAR$, implies that  $\mr{M}_{\geq n}$ are also a  sub $\cA^{\R}$-module of $\mr{M}$. Therefore, we get an $\cAR$-module filtration of $\mr{M}$
\[ 
\mr{M} = \mr{M}_{\geq k} \supset \mr{M}_{\geq k+1} \supset  \dots \supset \mr{M}_{\geq k+l} = {\bf 0} 
\]
such that we for each $i$ there is a short exact sequences 
\begin{equation} \label{SES-M}
\begin{tikzcd}
0 \rar & \mr{M}_{\geq i +1} \rar & \mr{M}_{\geq i} \rar & \underset{\{ b \in \mathcal{B}_{\mr{M}}: {\sf t}(b) = i \} }{\bigoplus} \Sigma^{|b|} \bM^\R \rar & 0
\end{tikzcd}
\end{equation}
of $\cAR$-modules.  

A short exact sequence of $\cAR$-modules gives a long exact sequence in $\Ext$. By splicing the long exact sequences induced by \eqref{SES-M}, we get an ``algebraic" Atiyah-Hirzebruch spectral sequence
\begin{equation} \label{aAHSS}
 \mr{E}_2^{s', w', s, f, w} :=  \mathcal{B}_{\mr{M}}^{s', w'} \otimes  \Ext_{\cAR}^{s,f,w}(\mr{M}, \bMR) \Rightarrow \Ext_{\cAR}^{s - s' ,f,w - w'}(\mr{M}, \mr{M})
\end{equation}
and a corresponding weak version of \autoref{TRT}, along with a uniqueness criterion, which is a weak form of \autoref{TUT}.

\begin{thm} \label{wTRT1} Let $\mr{M}$ denote an $\cAR$-module whose underlying $\bMR$-module is free and finite.  Suppose 
\[ \Ext_{\cAR}^{s-2,f,w}(\mr{M}, \bMR) = 0\]
for $f \geq 3$ whenever $(s,w)  \in  \mathcal{D}_{\mr{M}}$. Then there exists an $\rX \in \Sp^\R_{2, \mr{fin}}$ such that 
$ \rHR^{*,*}(\rX) \cong \mr{M}$ 
as an $\cAR$-module.  Further, such a realization is unique if 
\[ \Ext_{\cAR}^{s-1,f,w}(\mr{M}, \bM^\R) = 0\]
for  all $f \geq 2$ and  $(s,w)  \in  \mathcal{D}_{\mr{M}}$.
\end{thm}

\subsection{Weak $\R$-motivic  Toda realization -- version (II)}  \label{subsec:wTRT2}
\

For any $\cAR$-module $\mr{M}$ which is $\bMR$-free, the quotient $\mr{M}/(\rho)$ is an $\cA^{\C}$-module. In particular,  \[ \cAR/(\rho) \cong \cA^\C\] as a graded Hopf-algebra. Therefore, we have  a spectral sequence 
\begin{equation} \label{rhoBSS}
\begin{tikzcd} 
 {^{\rho}}\mr{E}_{2}^{s,f,w,i} := \bigoplus_{i \geq 0} \Ext_{\cA^\C}^{s+i,f, w+i}(\mr{M}/(\rho), \bM^{\C}) \rar[Rightarrow]  & \Ext_{\cAR}^{s,f,w}(\mr{M}, \bMR) 
 \end{tikzcd}
 \end{equation}
 which is often called the  \emph{(algebraic) $\rho$-Bockstein spectral sequence}.  
  Thus we get the following version of the $\R$-motivic Toda realization and uniqueness theorem which is weaker than \autoref{wTRT1}.  
  
 \begin{thm} \label{wTRT2} Let $\mr{M}$ denote an $\cAR$-module whose underlying $\bMR$-module is free and finite.  Suppose 
\[ \Ext_{\cA^\C}^{s-2+i,f,w+i}(\mr{M}/(\rho), \bM^\C) = 0\]
for $f \geq 3$ and all $i\geq0$ whenever  $(s,w)  \in  \mathcal{D}_{\mr{M}}$, then there exists an $\rX \in \Sp^\R_{2, \mr{fin}}$ such that 
$ \rHR^{*,*}(\rX) \cong \mr{M}$ 
as an $\cAR$-module. Further, such a realization is unique if 
\[ \Ext_{\cA^\C}^{s-1+i,f,w+i}(\mr{M}/(\rho), \bM^\C) = 0\]
for  all $f \geq 2$, $i \geq0$ and  $(s,w)  \in  \mathcal{D}_{\mr{M}}$.
   \end{thm}

\subsection{Weak $\R$-motivic Toda realization -- version (III)} \label{subsec:wTRT3}
\

{Similarly to the classical case, the $\C$-motivic Steenrod algebra enjoys an  increasing filtration called the May filtration,  which is easier to express on its  dual (see \cite{LowMilnorWitt}). 
On $\cA^\C_*$, the May filtration  is  induced by  assigning the May weights }
\[ \may(\tau_{i-1}) = \may(\upxi_{i}^{2^j}) = 2i -1 \]
and extending it multiplicatively. The associated graded is an exterior algebra 
\begin{equation} \label{xi-ij}
{\gr( \cA^\C)} \cong \Lambda_{\bM^\C}(\upxi_{i,j} : i \geq 1, j \geq 0), 
\end{equation}
where $\upxi_{i,0}$ represents $(\tau_{i-1})_*$ and $(\upxi_{i,j+1})_*$ represents $(\upxi_i^{2^j})_*$ in the associated graded. When $\mr{M} = \bM^\R$ in \eqref{eqn:maySS}, then 
\[ {^{\mr{May}}}\mr{E}_{1, \bM^\C }^{*,*,*,*} \cong \bM^\C[\sfh_{i,j}: i \geq 1, j \geq 0] , \]
 where $\sfh_{i,j}$ represents the class $\upxi_{i,j}$. The  $(s, f, w, \may)$-degrees of these generators are given by  
\[ |\sfh_{i,j}| =  \left\lbrace \begin{array}{cll} 
(2^{i} -2,1,2^{i-1} -1,i) & \text{if $j =0$, and, } \\
 (2^j(2^{i} -1)-1, 1,2^{j-1}(2^i-1) ,i) & \text{otherwise.}
\end{array} \right.
 \]

When $\mr{M}$ is a cyclic  $\cAR$-module, $\mr{M}/(\rho)$ is also cyclic as an $\cA^\C$-module, thus the May filtration induces a filtration on $\mr{M}/(\rho)$. Thus, we get a  corresponding May spectral sequence 
\begin{equation} \label{eqn:maySS}
 {^{\mr{May}}}\mr{E}_{1, \mr{M}/(\rho)}^{s,f,w,\may} := \Ext_{\gr(\cA^\C)}^{s,f,w,m}(\gr (\mr{M}/(\rho)), \bM^\C) \Rightarrow \Ext_{\cA^\C}^{s,f, w}(\mr{M}/(\rho), \bM^{\C})
 \end{equation}
computing the input of the $\rho$-Bockstein spectral sequence \eqref{rhoBSS}. Thus we can formulate a  version  of $\R$-motivic Toda realization theorem which is  even weaker than \autoref{wTRT2}.

\begin{thm} \label{wTRT3} Let $\mr{M}$ denote an cyclic $\cAR$-module whose underlying $\bMR$-module is free and finite.  Suppose 
\[ {^{\mr{May}}}\mr{E}_{1,  \mr{M}/(\rho)}^{s-2+i,f,w+i, *}   = 0.\]
for $f \geq 3$ and all $i\geq0$ whenever  $(s,w)  \in  \mathcal{D}_{\mr{M}}$. Then there exists an $\rX \in \Sp^\R_{2, \mr{fin}}$ such that 
$ \rHR^{*,*}(\rX) \cong \mr{M}$ 
as an $\cAR$-module. Further, such a realization is unique if 
\[ {^{\mr{May}}}\mr{E}_{1,  \mr{M}/(\rho)}^{s-1+i,f,w+i, *}   = 0\]
for $f \geq 2$, $i \geq0$ and  $(s,w)  \in  \mathcal{D}_{\mr{M}}$.
 \end{thm}

\section{A comparison between $\Ctwo$-equivariant and classical squaring operations}
\label{sec:Comparison}

For any $\Ctwo$-equivariant space $\rX \in \Top_*^{\Ctwo}$  we can functorially assign two non-equivariant spaces -- the underlying space $\ParRes{\rX}$, which is obtained by restricting the action of  $\Ctwo$ to the trivial group,  and the space of $\Ctwo$-fixed-points $\Fix{\rX}$.  For a $\Ctwo$-equivariant spectrum $\mr{E} \in \Sp^{\Ctwo}$, restricting the action to the trivial subgroup results in a monoidal functor 
\[ 
\begin{tikzcd}
\Res: \Sp^{\Ctwo} \rar & \Sp
\end{tikzcd}
\]
 that identifies the underlying spectrum. However, there are two different notions of fixed-point spectrum -- the categorial fixed-points and the geometric fixed-points.  
 
   The categorical fixed-points functor  is a lax monodial functor
 \[ 
\begin{tikzcd}
( -)^\Ctwo: \Sp^{\Ctwo} \rar & \Sp, 
\end{tikzcd}
\]
which is defined so that  $\pi_k(\CFix{\mr{E}}) \cong  \pi_{k}^{\Ctwo}(\mr{E})$, but it does not interact well  with infinite suspensions. The correction term is explained by the  tom~Dieck splitting  \cite{LMS}*{Theorem~V.11.1}: 
\begin{equation} \label{TomD}
\CFix{(\Sigma^\infty_{\Ctwo}\rX)} \simeq \Sigma^\infty (\CFix{\rX}) \vee  \Sigma^\infty (\rX_{\mr{h} \Ctwo}),
\end{equation}
where $\rX_{\mr{h} \Ctwo}$ is the homotopy orbit space. Let  $\widetilde{\mr{E}\Ctwo}:= \mr{Cof}( \mr{E}\Ctwo_+ \to \bS )$. The geometric fixed-point functor 
\[
\begin{tikzcd}
\GFixFunctor: \Sp^{\Ctwo} \rar & \Sp,
\end{tikzcd}
\]
is a symmetric monoidal functor given by $\GFix{\mr{E}} := \CFix{(\mr{E} \sma \widetilde{\mr{E}\Ctwo})}$.  When $\mr{E} \in \Sp^{\Ctwo}$, 
\begin{equation} \label{GFixsuspension} 
\GFix{\Sigma^\infty_{\Ctwo}\mr{E}} \simeq  \Sigma^{\infty}\Fix{\mr{E}}
\end{equation}
is the  first component  in  \eqref{TomD}. For any $\mr{E} \in \Sp^{\Ctwo}$, there is a natural map of spectra 
\[ 
\begin{tikzcd}
\upiota_{\mr{E}}: \CFix{\mr{E}} \rar & \GFix{\mr{E}} 
\end{tikzcd}
\]
 induced by the map $\bS \to \widetilde{\mr{E}\Ctwo}$.

 The Eilenberg-Mac~Lane spectrum $\rH\uF$ is an $\bE_\infty^{\Ctwo}$-ring (\cite{LMS}*{VII}), i.e. a commutative monoid as a genuine $\Ctwo$-spectrum. The restriction $ \ParRes{\rH\uF} \simeq \rH\F$, the categorical fixed-points $\CFix{\rH\uF} \simeq \rH\F$ and the geometric fixed-points $\GFix{\rH\uF} \simeq  \rH\F[t]$  are $\bE_\infty$-rings. 
 It follows from the knowledge of $\bM^\Ctwo := \pi_{\star}^\Ctwo \rH\uF$ that 
\[ 
\begin{tikzcd}
\CFix{(\Sigma^{n \upsigma} \rH\uF)} \simeq \bigvee_{i=0}^{n} \Sigma^i \rH\F   \rar[hook] & \GFix{\rH\uF} \simeq \mr{colim}_n \  \CFix{(\Sigma^{n \upsigma} \rH\uF)} \simeq \rH\F[t]
\end{tikzcd}
\]
is the inclusion of the first $(n+1)$ components. The above map clearly splits. One can endow $\CFix{(\Sigma^{n \upsigma} \rH\uF)}$ with an $\bE_{\infty}$-structure isomorphic to  the truncated polynomial algebra $\rH\F[t]/(t^{n+1})$ so that the splitting map 
\[ 
\begin{tikzcd}
\uppi_{\F}^{(n)}:\GFix{\rH\uF} \simeq \rH\F[t] \rar[two heads] &\CFix{(\Sigma^{n \upsigma} \rH\uF)}  \simeq  \rH\F[t]/(t^{n+1})
\end{tikzcd}
\]
is an $\bE_{\infty}$-map. The composition
\begin{equation} \label{splittingHF}
\begin{tikzcd}
\CFix{\rH\uF} \rar[hook, "\upiota_{\F}"] & \GFix{\rH\uF} \rar[two heads, "\uppi_{\F}^{(0)}"] & \CFix{\rH\uF}
\end{tikzcd}
\end{equation}
is the identity and  exhibits  $\GFix{\rH\uF}$ as an augmented $\rH\F$-algebra.

For any $\Ctwo$-space $\rX \in \Top^\Ctwo_*$, the restriction functor induces a natural transformation 
\[ 
\begin{tikzcd}
\Res_*: \rHC^{i,j}(\rX_+)  \rar & \rH^{i}(\ParRes{\rX}_+).
\end{tikzcd}
\]
To compare the cohomology of $\rX^{\Ctwo}$ with the $\RO(\Ctwo)$-graded cohomology of $\rX$, we make use of the splitting \eqref{splittingHF}
to define the natural ring map
\[ 
\begin{tikzcd}
\modPhi: \rHC^{i,j}(\rX_+) \rar & \rH^{i-j}(\rX^{\Ctwo}_+),
\end{tikzcd}
\]
which sends $u \in \rHC^{i,j}(\rX_+)$ to the composite (as defined in \cite{BW}*{(2.7)})
\[ 
\begin{tikzcd}
  \Sigma^{\infty}\rX^{\Ctwo}  \ar[rr,"\GFix{u}"] && \Sigma^{i-j}\GFix{ \rH\uF} \ar[r, two heads, "\uppi_{\F}^{(0)}"] &\Sigma^{i-j} \rH\F .
\end{tikzcd}
\]

The purpose of this section is to compare the $\RO(\Ctwo)$-graded squaring operations with the classical squaring operations along the maps $\Res_*$ and $\modPhi$. We begin with a brief recollection of the construction of the classical and $\Ctwo$-equivariant squaring operations. 

\subsection{Steenrod's construction of squaring operations}

 The construction of the classical mod $2$ Steenrod algebra, which is the algebra of stable cohomology operations for ordinary cohomology with $\F$-coefficients,  involves  the $\mathbb{E}_\infty$-structure\footnote{Technically, we only make use of the  $\mathbb{H}_\infty$-ring structure that underlies the $\bE_\infty$-structure of $\rH\F$} of $\rH\F$  and the fact that the tautological line bundle $\upgamma$ over $\RP^\infty$ is $\rH\F$-orientable. 
 We review here how the mod 2 Steenrod operations are derived from that structure. A similar discussion can be found in \cite{Hinf}*{Section~VIII.2}.

 \begin{notn}
 For any space or spectrum $\rX$ and $n\geq 1$, we let \[ \mr{D}_n(\rX) := (\mr{E}\Sigma_n)_+ \sma_{\Sigma_n} (\rX^{\sma n}),\]
 where $\Sigma_n$ acts by permuting the factors of $\rX^{\sma n}$. By convention, $\mr{D}_0(\rX)= \bS$.
 \end{notn}
 An $\mathbb{E}_\infty$-ring structure on a spectrum $\mr{R}$ is a collection of maps   of the form 
 \[ 
\begin{tikzcd}
\Theta_n^{\mr{R}} : \mr{D}_n(\mr{R}) \rar & \mr{R}
 \end{tikzcd}
 \]
 for each $n \geq 0$, which satisfy the usual coherence conditions (see \cite{GeomIter}). By assumption, $\Theta_0^{\mr{R}}$ is the unit map of $\mr{R}$ and $\Theta_1^{\mr{R}}$ is the identity map.

  The $\rH\F$-orientibility of $\upgamma$ implies the existence of an $\rH\F$-Thom class 
 \begin{equation}
 \begin{tikzcd}
\uu_n: \mr{Th}(\upgamma^{\oplus n}) \simeq \RP_n^\infty \rar & \Sigma^n\rH\F
\end{tikzcd}
 \end{equation}
for each $n \geq 0$. These are compatible as $n$ varies, in the sense that the following diagram commutes:
\begin{equation} \label{thomcompat}
 \begin{tikzcd}
 \mr{Th}(\upgamma^{\oplus (m+n)}) \rar \ar[rrd, bend right = 10, "\uu_{m+n}"'] & \mr{Th}(\upgamma^{\oplus m}) \sma \mr{Th}(\upgamma^{\oplus n}) \rar["\uu_m \sma \uu_n"] & \Sigma^m\rH\F \sma \Sigma^n\rH\F \dar["\upmu_{\F}"]  \\
&& \rH\F.
\end{tikzcd}
\end{equation}

For any spectra $\mr{E}$ and $\mr{F}$, there is a natural map 
\[ 
\begin{tikzcd}
\updelta_n: \mr{D}_n(\mr{E} \sma \mr{F}) \rar[""] & \mr{D}_n(\mr{E}) \sma \mr{D}_n(\mr{F})
\end{tikzcd}
\]
induced by the diagonal on $\mr{E}\Sigma_n$ and the  isomorphism $(\mr{E}\sma \mr{F})^{\sma n} \iso \mr{E}^{\sma n} \sma \mr{F}^{\sma n}$.
Thus, we may define the map $\uptau_n$ as the composition  
\begin{equation}
 \begin{tikzcd}[column sep={37}]
 \mr{D}_2(\Sigma^n\rH\F) \ar[d,"\updelta_2"]  \ar[rr,"\uptau_n"] &&\Sigma^{2n} \rH\F \\
 \mr{D}_2(\mr{S}^n) \sma \mr{D}_2(\rH\F) \rar["\simeq"]   & \Sigma^n\RP_n^\infty \sma \mr{D}_2(\rH\F)\ar[r," \Sigma^n\uu_n \sma \Theta_2^{\F}"] & \Sigma^{2n} \rH\F \sma \rH\F. \uar["\upmu_{\F}"]
\end{tikzcd} 
 \end{equation} 
 
  \begin{defn}
  The power operation  is a natural  transformation
 \[ 
\begin{tikzcd}
\Pow_{2} \colon \rH^{n}(-) \rar & \rH^{2n }(\mr{D}_2(-)), 
\end{tikzcd}
\]
which takes a class $u \in \mr{H}^{n}( \mr{E} )$ to the composite class
\[ 
\begin{tikzcd}
\Pow_{2}(u)\colon \mr{D}_2(\mr{E}) \rar["\mr{D}_2(u)"]  & \mr{D}_2(\Sigma^{n}\rH\F) \rar["{\uptau}_n"] & \Sigma^{n} \rH\F
\end{tikzcd}
\]
for any $\mr{E} \in \Sp$.  
\end{defn}

From  \eqref{thomcompat}, we deduce the commutativity  of the diagram
 \begin{equation} \label{cartan}
 \begin{tikzcd}
\mr{D}_2(\Sigma^n \rH\F \sma \Sigma^m \rH\F)\rar \dar["\mr{D}_2(\upmu_{\F})"] &\mr{D}_2(\Sigma^{n}\rH\F)  \sma  \mr{D}_2(\Sigma^{m}\rH\F) \rar["\uptau_{n} \sma \uptau_{m}"]  & \Sigma^{2n}\rH\F \sma \Sigma^{2m}\rH\F \dar["\upmu_{\F}"]  \\
 \mr{D}_2(\Sigma^{n+m}\rH\F) \ar[rr,"\uptau_{n+m}"'] &&\Sigma^{2n+2m} \rH\F. 
 \end{tikzcd}
 \end{equation}
 As a result, we have 
\[ \updelta_2^*(\Pow_2(u) \otimes \Pow_2(v)) = \Pow_{2}(u \otimes  v) \]
which leads to the  Cartan formula for the Steenrod algebra. 

If $\rX \in \Top_*$ is given the trivial $\Sigma_2$-action and $\rX \sma \rX$  the permutation action, the  diagonal map $ \rX \to \rX \sma \rX$  is $\Sigma_2$-equivariant. Consequently, we have an induced map 
\[ 
\begin{tikzcd}
\Delta_{\mr{X}}: (\mr{B}\Sigma_2)_+ \sma \rX  \simeq (\mr{E}\Sigma_2)_+ \sma_{\Sigma_2} \rX \rar[] & \mr{D}_2(\rX).
\end{tikzcd}
\]
Since $\rH^*(\mr{B}\Sigma_2)\cong \F[\ttt]$, we may write (using the Kunneth isomorphism)
 \begin{equation} \label{formula:classic}
  \Delta_{\mr{X}}^*(\Pow_2(u)) = \sum_{i = 0}^{n} \ttt^{n-i} \otimes \Sq^i(u),
  \end{equation}
  which defines the natural transformations $\Sq^i\colon \rH^n(-) \rtarr \rH^{n+i}(-)$.

  \begin{rmk} \label{rmk:suspSq} The squaring operation $\Sq^i(u)$ for any class $u \in \rH^n(\rX)$  is determined by  $\Sq^i(\iota_n)$, where $\iota_n \in \rH^*(\mr{K}(\F, n))$ is the fundamental class, because of the universal property of $\mr{K}(\F, n)$. A priori, $\Sq^i(u)$ depends on the cohomological degree of $u$. However, this dependence is eradicated by the fact that the squaring operations are stable, i.e.  for any $u \in \rH^*(\rX)$
  \[ \Sq^i(\sigma_*(u)) = \sigma_*(\Sq^i(u)), \]
  where $\sigma_*: \rH^{*}(\rX) \cong \rH^{\ast +1}(\Sigma \rX)$ is the suspension isomorphism.  The  $\rH\F$-orientibility of  $\upgamma$ implies $\Sq^0(\iota) =\iota$ for the generator $\iota \in \rH^1(\mr{S}^1)$, which, along with Cartan formula, implies stability.
  \end{rmk}
  
\subsection{ The $\Ctwo$-equivariant squaring operations}
The construction of the classical squaring operations can be adapted to construct  squaring operations on the $\RO(\Ctwo)$-graded cohomology of a $\Ctwo$-space. 

\begin{rmk} Our ideas are closely related to the construction of the $\R$-motivic squaring operations  due to Voevodsky \cite{V}. Certain parts, such as the construction of the power operation \autoref{defn:power}, though different, can be compared to \cites{W1,W2},  where the author studies $\Ctwo$-equivariant power operations on the homology of spaces.  
 \end{rmk}
 
\begin{notn} For any  group $\mr{G}$ and a family of subgroups $\mathcal{F}$ closed under subconjugacy, there exists a space $\mr{E}\mathcal{F}$ 
determined up to a $\mr{G}$-weak equivalence by its universal property 
\[ 
 \mr{E} \mathcal{F}^{\mr{H}} \simeq \left\lbrace \begin{array}{ccccc}
 \ast & \text{if $\mr{H} \in \mathcal{F},$} \\
 \emptyset & \text{otherwise.}
  \end{array} \right.
 \]
When $\mr{G} = \Ctwo \times \Sigma_n$ and $\mathcal{F}_n = \{ \rH \subset \mr{G}: \rH \cap \Sigma_n = \mathbbm{1} \}$, we denote $\mr{E}\mathcal{F}_n$ by $\mr{E}_{\Ctwo}\Sigma_n$.
Note that there is a natural $\Ctwo$-equivariant map
 $ \mr{E}\Sigma_n \rtarr \mr{E}_\Ctwo \Sigma_n$.
\end{notn}

\begin{notn}
 For a based $\Ctwo$-space or a $\Ctwo$-spectrum $\rX$, we let  \[\mr{D}_n^{\Ctwo}(\rX) := (\mr{E}_{\Ctwo}\Sigma_n)_+ \sma_{\Sigma_n} (\rX^{ \sma n}) \]
 the $n$-th equivariant extended power construction on $\rX$.  There is a natural $\Ctwo$-equivariant map
 \[ 
 \begin{tikzcd}
 \updelta_n^{\Ctwo} : \mr{D}_n^{\Ctwo}(\rX \sma \mr{Y}) \rar &  \mr{D}_n^{\Ctwo}(\rX) \sma \mr{D}_n^{\Ctwo}(\mr{Y}) 
 \end{tikzcd}
  \]
  induced by the diagonal map of $\mr{E}_{\Ctwo}\Sigma_n$ for any pair $\rX$ and $\mr{Y}$ of $\Ctwo$ space or spectra.  
\end{notn}

For a $\Ctwo$-equivariant space $\rX \in \Top_*^{\Ctwo}$, the inclusions $\rX^\Ctwo \hookrightarrow \rX$ and $\mr{E}\Sigma_n \rtarr \mr{E}_\Ctwo \Sigma_n$ together induce a natural map  
\begin{equation} 
\begin{tikzcd}
\uplambda_{\rX}: \mr{D}_2(\Fix{\rX}) \rar & \Fix{\mr{D}_2^{\Ctwo}(\rX)}
\end{tikzcd}
\end{equation}
which is usually not an equivalence.

\begin{eg} \label{maplambdasphere}
When $\rX \simeq \mr{S}^0$, 
$
\begin{tikzcd}
\uplambda_{\mr{S}^0}:(\mr{B}\Sigma_2)_+ \rar & \Fix{(\mr{B}_{\Ctwo}\Sigma_2)} \simeq \mr{B}\Sigma_2 \sma \mr{S}^0_+
\end{tikzcd}
$
is the inclusion of a summand. 
\end{eg}

 Likewise, when  $\mr{E}\in \Sp^{\Ctwo}$, the map $\mr{E}^{\Ctwo} \hookrightarrow \mr{E}$  induces a natural map 
\[ 
\begin{tikzcd}
\uplambda_{\mr{E}}: \mr{D}_2(\Fix{\mr{E}}) \rar & \Fix{\mr{D}_2^{\Ctwo}(\mr{E})}.
\end{tikzcd}
\]
 Using the fact that $\widetilde{\mr{E}\Ctwo}$ is an $\bE_{\infty}$-ring $\Ctwo$-spectrum  
 we define a map $\uplambda^{\Phi}_{\mr{E}}$ as the composition
\begin{equation}
\begin{tikzcd}
\mr{D}_2 (\GFix{\mr{E}}) \ar[rr,"\uplambda^{\Phi}_{\mr{E}}"] \dar["\lambda_{\widetilde{\mr{E}\Ctwo}\sma \mr{E}}"] & & \GFix{\mr{D}_2^{\Ctwo}(\mr{E}) } \\
(\mr{D}_2(\widetilde{\mr{E}\Ctwo} \sma \mr{E}))^{\Ctwo} \rar
& \CFix{ (\mr{D}_2(\widetilde{\mr{E}\Ctwo})\sma \mr{D}_2(\mr{E}))}  \rar  &\CFix{( \widetilde{\mr{E}\Ctwo} \sma  \mr{D}_2 ( \mr{E}))} \cong \GFix{\mr{D}_2(\mr{E})} \uar \\
\end{tikzcd}
\end{equation} 
 By definition, an $\mathbb{E}_\infty^{\Ctwo}$-ring structure on a spectrum $\mr{R}$ consists of  a system of maps 
 \[ 
 \begin{tikzcd}
 \Theta_{n}^{\mr{R}}: \mr{D}_n^{\Ctwo}(\mr{R}) \rar & \mr{R}
 \end{tikzcd}
 \]
 for each $n \geq 0$, which satisfy certain compatibility criteria \cite{LMS}*{\S{VII.2}}.
 The categorical fixed-point spectrum $\CFix{\mr{R}}$ as well as the geometric-fixed point spectrum $\GFix{\mr{R}}$ of an $\mathbb{E}_\infty^{\Ctwo}$-ring spectrum $\mr{R}$ are $\mathbb{E}_\infty$-ring spectra with structure maps  
 \[ 
  \begin{tikzcd}
 \Theta_{n}^{\CFix{\mr{R}}}: \mr{D}_2(\CFix{\mr{R}}) \rar["\lambda_\mr{R}"] & \CFix{\mr{D}_2^{\Ctwo}(\mr{R})} \ar[rr, "\CFix{(\Theta_{n}^{\mr{R}})}"] && \CFix{\mr{R}}
 \end{tikzcd}
 \]
 and
 \[ 
  \begin{tikzcd}
 \Theta_{n}^{\GFix{\mr{R}}}: \mr{D}_2(\GFix{\mr{R}}) \rar["\lambda_\mr{R}^\Phi"] & \GFix{\mr{D}_2^{\Ctwo}(\mr{R})} \ar[rr, "\GFix{\Theta_{n}^{\mr{R}}}"] && \GFix{\mr{R}},
 \end{tikzcd}
 \]
 respectively. Further, the natural map 
 \[ 
 \begin{tikzcd}
 \upiota_{\mr{R}} : \CFix{\mr{R}} \rar & \GFix{\mr{R}}
 \end{tikzcd}
 \]
is an  $\bE_\infty$-ring map.

 Let  $\upomega$ denote the sign representation of $\Sigma_2$.  The equivariant Eilenberg-Mac~Lane spectrum $\rH\uF$ does not distinguish between the $\Ctwo$-equivariant bundles 
\[ 
 \begin{tikzcd}
 \overline{\upepsilon}:  \mr{E}_{\Ctwo}\Sigma_2 \times_{\Sigma_2} (\uprho ) \rar & \mr{B}_{\Ctwo} \Sigma_2
 \end{tikzcd}
 \]
 \[ 
 \begin{tikzcd}
 \overline{\upgamma}:  \mr{E}_{\Ctwo}\Sigma_2 \times_{\Sigma_2} (\uprho \otimes \upomega) \rar & \mr{B}_{\Ctwo} \Sigma_2,
 \end{tikzcd}
 \]
i.e. there exists a $\Ctwo$-equivariant Thom isomorphism 
\[ \mr{Th}(\overline{\upgamma}) \sma \rH\uF \simeq  \mr{Th}(\overline{\upepsilon}) \sma \rH\uF \simeq \Sigma^{\uprho} (\mr{B}_{\Ctwo} \Sigma_2)_+ \sma   \rH\uF.   \]
The above Thom isomorphism results in an $\rH\uF$-Thom class  
\[
 \begin{tikzcd} 
 \underline{\uu}_n: \mr{Th}(\overline{\upgamma}^{\oplus n}) \rar & \Sigma^{n\rho} \rH\uF 
 \end{tikzcd}
   \]
   for each $n \geq 0$, and these Thom classes can be used to define the $\Ctwo$-equivariant power operations. Since \[ \mr{D}_2^{\Ctwo}(\mr{S}^{n \uprho}) \simeq \mr{Th}(n \uprho \oplus n (\uprho \otimes \upomega)) \simeq \Sigma^{n\uprho} \mr{Th}(\overline{\upgamma}^{\oplus n}),\] we define the map $\underline{\uptau}_n$ as the composition
\begin{equation}
 \begin{tikzcd}[column sep={37}]
 \mr{D}_2(\Sigma^{n\rho}\rH\uF) \ar[d,]  \ar[rr,"\underline{\uptau}_n"] &&\Sigma^{2n\rho} \rH\uF \\
 \mr{D}_2^\Ctwo(\mr{S}^{n\rho}) \sma \mr{D}_2^\Ctwo(\rH\uF) \rar["\simeq"]   &  \Sigma^{n\uprho} \mr{Th}(\overline{\upgamma}^{\oplus n}) \sma \mr{D}_2^\Ctwo(\rH\uF)\ar[r," \Sigma^n\underline{\uu}_n \sma \Theta_2^{\uF}"] & \Sigma^{2n\rho} \rH\uF \sma \rH\uF. \uar["\upmu_{\F}"]
\end{tikzcd} 
 \end{equation}

 \begin{defn} \label{defn:power} 
 The equivariant power operation  is a natural  transformation
 \[ 
\begin{tikzcd}
\Pow_{n\uprho}: \rHC^{n \uprho}(-) \rar & \rHC^{2n \uprho}(\mr{D}_2^{\Ctwo}(-)), 
\end{tikzcd}
\]
which takes a class $u \in \rHC^{n \uprho}( \rE )$ to the composite class
\[ 
\begin{tikzcd}
\Pow_{n \uprho}(u): \mr{D}_2^{\Ctwo}(\rE) \rar["\mr{D}_2^{\Ctwo}(u)"]  & \mr{D}_2^{\Ctwo}(\Sigma^{n\uprho}\rH\uF) \rar["\underline{\uptau}_n"] & \Sigma^{2n\uprho} \rH\uF
\end{tikzcd}
\]
for any $\rE \in \Sp^{\Ctwo}$. 
\end{defn}

When $\rX \in \Top_*^\Ctwo$ is given the trivial $\Sigma_2$-action and $\rX \sma \rX$ is given the permutation action, the  diagonal map $ \rX \to \rX \sma \rX$  is a $\Ctwo \times \Sigma_2$-equivariant map. Consequently, we have a $\Ctwo$-equivariant map 
\[ 
\begin{tikzcd}
\Delta^{\!\Ctwo}_{\rX}: (\mr{B}_{\Ctwo}\Sigma_2)_+ \sma \rX  \simeq (\mr{E}_{\Ctwo}\Sigma_2)_+ \sma_{\Sigma_2} \rX \rar[] & \mr{D}_2^{\Ctwo}(\rX).
\end{tikzcd}
\]
By \cite{HK}*{Lemma~6.27} (also see \cite{W1}*{Proposition~3.2}), 
\[ \rHC^{\star}((\mr{B}_{\Ctwo}\Sigma_2)_+)\cong \bMC[\yy,\xx]/(\yy^2 = a_\upsigma \yy + u_{\upsigma} \xx),\] where $|\yy| = (1,1)$ and $|\xx|= (2,1)$. 
Since $\rHC^{\star}((\mr{B}_{\Ctwo}\Sigma_2)_+)$ is $\bMC$-free, we also have a Kunneth isomorphism 
\[ \rHC^{\star}((\mr{B}_{\Ctwo}\Sigma_2)_+ \sma \rX) \cong \rHC^{\star}((\mr{B}_{\Ctwo}\Sigma_2)_+) \otimes_{\bMC} \rHC^{\star}\!(\rX).\]
 Thus, for any $u \in \rHC^{n \uprho}(\rX)$, we may write $(\Delta^{\!\Ctwo}_{\rX})^*(\mathcal{P}_{n \uprho}(u))$ using the formula
 \begin{equation} \label{formula:equiv}
  (\Delta^{\!\Ctwo}_{\rX})^*(\mathcal{P}_{n \uprho}(u)) = \sum_{i = 0}^{n} \xx^{n-i} \otimes \eSq^{2i}(u) + \sum_{i = 0}^{n}\yy \xx^{n-i-1} \otimes \eSq^{2i+1}(u), 
  \end{equation}
   which defines the equivariant squaring operations $\eSq^{i}$  for all $i \geq 0$. These can be extended to operations on the entire $\RO(\Ctwo)$-graded cohomology ring as in \cite{V}*{Prop~2.6}).
\begin{rmk}   Just like  the classical case, one can  easily deduce that the $\RO(\Ctwo)$-graded squaring operations defined this way are natural, stable  and obey the Cartan formula. In fact, Voevodsky \cite{V}  uses a similar approach to establish these  properties for the $\R$-motivic Steenrod algebra, which can be emulated in the $\Ctwo$-equivariant case using the Betti realization functor.
\end{rmk}
   
\subsection{Comparison theorems} 
 Since the restriction functor is  monoidal, it induces a ring map 
\[ 
\begin{tikzcd}
\Res_*: \rHC^{\star}\!(\rX_+) \rar & \rH^{\ast}(\ParRes{\rX}_+)
\end{tikzcd}
\]
for any $\rX \in \Top^{\Ctwo}_*$. 

\begin{eg}\label{eg:ResHF}
 When $\rX = \ast$, 
the map 
\[ 
\begin{tikzcd}
\Res_*:  \pi_{\star}^{\Ctwo}\rH\uF  \rar & \pi_* \rH\F \cong \F
\end{tikzcd}
\]
sends  $u_\upsigma \mapsto 1$, $a_\upsigma \mapsto 0$, and $\Theta \mapsto 0$. This  follows from the fact that the cofiber sequence $\Ctwo_+ \rtarr S^0 \xrtarr{a_\upsigma}S^\upsigma$ shows that the kernel of $\Res_*$ consists of precisely the $a_\upsigma$-divisible elements.
\end{eg}

\begin{prop} For  any $\rX \in \Top_*^{\Ctwo}$ and a class $u \in \rHC^{n \uprho}(\rX)$ 
\[\Res_*(\mathcal{P}_{n \uprho}(u)) = \mathcal{P}_{2n}( \Res_*(u)).\]
\end{prop}

\begin{proof} Since, $\Res(\overline{\upgamma}) = 2 \upgamma$, it follows that  $\Res_*(\underline{\uu}_n) = \uu_{2n}$. This, along with the fact that $\ParRes{\Theta_2^{\uF}} =\Theta_2^{\F}$ shows $\Res(\underline{\uptau}_n) = \uptau_{2n}$, and the result follows. 
\end{proof}

\begin{proof}[Proof of \autoref{thm:squareunder}]  
Let $\rX \in \Top_*^{\Ctwo}$ and $u \in \rHC^{n \uprho}(\rX)$. Since $\ParRes{\mr{B}_{\Ctwo}\Sigma_2} \simeq \mr{B}\Sigma_2$, $\ParRes{\Delta^{\Ctwo}} = \Delta$, 
$\Res_*(\yy) = \ttt$ and $\Res_*(\xx) = \ttt^2$, it follows that 
\begin{eqnarray*} 
\Delta_{\ParRes{\rX}}^*( \mathcal{P}_{2n}(\Res_*(u))) &=&\sum_{i = -n}^{n} \ttt^{n-i} \otimes \Sq^i(\Res_*(u)) \\
\end{eqnarray*}
must equal 
\begin{eqnarray*} 
{\Res_*( (\Delta_\rX^\Ctwo)^*(\mathcal{P}_{n \uprho }(u)) )}&=&  \Res_*(\sum_{i = -n}^{n} \xx^{n-i} \otimes \eSq^{2i}(u) + \sum_{i = -n}^{n}\yy \xx^{n-i-1} \otimes \eSq^{2i+1}(u)  ) \\
&=& \sum_{i = -n}^{n} \ttt^{2n-2i} \otimes\Res_*(\eSq^{2i}(u))  \\
&& \hspace{80pt} + \sum_{i = -n}^{n}\ttt^{2n-2i-1} \otimes \Res_*(\eSq^{2i+1}(u)).  
\end{eqnarray*}
Thus, the result is true for cohomology classes $u \in \rHC^{n \uprho}(\rX)$ for any space $\rX \in \Top_*^{\Ctwo}$. 

Since the squaring operations are stable, the result extends to arbitrary $\RO(\Ctwo)$-graded cohomology classes. Moreover, since $\Re$ commutes with suspensions, in the sense that
 $\Res\circ {\Sigma_{\Ctwo}^{\infty}} \simeq \Sigma^{\infty}\circ \Res$,
and any $\mr{E}\in \SpfinC$ is equivalent to $\Sigma^{-n} \Sigma^\infty_\Ctwo \rX$ for some $n$ and $\rX\in \Top_*^\Ctwo$, we conclude the same for any $u \in \rHC^{\star}\!(\mr{E})$. 
 \end{proof}

Now we draw our attention towards comparing the action of the $\Ctwo$-equivariant Steenrod algebra $\cA^\Ctwo$ on $\rH^{\star}(\rX_+)$ to the action of  the classical Steenrod algebra $\cA$ on $\rH^{*}(\Fix{\rX}_+)$, where $\rX \in \Top_*^{\Ctwo}$. Note that 
\[ 
\begin{tikzcd}
\modPhi:\rHC^\star\! (\rX_+) \rar & \rH^{*}(\Fix{\rX}_+) 
\end{tikzcd}
\]
is a ring map. 

\begin{eg}\label{eg:FixHF} 
When $\rX= \ast$, the map 
\[ 
\begin{tikzcd}
\modPhi \colon \pi_{\star}^{\Ctwo}\rH\uF \cong 
 \F[u_{\upsigma}, a_{\upsigma}] \oplus \Theta\{ u_{\upsigma}^{-i} a_{\upsigma}^{-j}\}
\rar & \pi_* \rH\F \cong \F
\end{tikzcd}
\]
sends $a_\upsigma \mapsto 1$, $u_\upsigma \mapsto 0$, and $\Theta\mapsto 0$. This is essentially because smashing with 
\[ \widetilde{\mr{E} \Ctwo} \simeq \mr{colim} \{ \mr{S}^0 \overset{a_{\upsigma}}\longrightarrow\mr{S}^{\upsigma} \overset{a_{\upsigma}}\longrightarrow \mr{S}^{2\upsigma}  \longrightarrow \dots \} \]
amounts to inverting $a_{\upsigma}$ and the projection $\uppi_{\F}^{(0)}$ kills $u_{\upsigma}$. 
\end{eg}

\begin{rmk} \label{rmk:Phiproj} One can deduce from \autoref{maplambdasphere} that in cohomology, the map 
\[ 
\begin{tikzcd}
\uplambda_{\mr{S}^0}^*: \rH^*(\Fix{\mr{B}_{\Ctwo}\Sigma_2}_+ ) \cong \F[\ttt][\iota]/(\iota^2 -\iota)  \rar[two heads] & \rH^*((\mr{B}\Sigma_2)_+) \cong \F[\ttt],
\end{tikzcd}
\] 
is the quotient map sending $\iota \mapsto 0$. 
\end{rmk}

\begin{eg} The map $\modPhi \colon \rHC^\star((\mr{B}_{\Ctwo} \Sigma_2)_+) \to \rH^{\ast}(\Fix{ \mr{B}_{\Ctwo} \Sigma_2}_+)$ sends 
$\xx \mapsto \ttt$ and $\yy \to \iota$, $a_{\upsigma} \mapsto 1$ and $u_\upsigma \mapsto 0$.
\end{eg}

\begin{lem}  \label{lem:eulercompare} The composition
 \[
 \begin{tikzcd}
   \rHC^{\star}\!(\mr{Th}(\overline{\upgamma}^{\oplus n}))  \rar["\modPhi"] & \rH^{\ast}(\Fix{\mr{Th}(\overline{\upgamma}^{\oplus n})}) \rar["\uplambda_{\mr{S}^{\uprho \otimes \upomega}}"] & \rH^*(\mr{Th}(\upgamma^{\oplus n}))
   \end{tikzcd}
   \]
sends $\underline{\uu}_n \mapsto \uu_n$.
\end{lem}

\begin{proof} Let $\upzeta_{\Ctwo}: (\mr{B}_\Ctwo\Sigma_2)_+ \to \mr{Th}(\overline{\upgamma}^{\oplus n})$ denote the zero-section. Under the zero section map the Thom class is mapped to the Euler class, and therefore $\upzeta_{\Ctwo}^*(\underline{\uu}_n) = \xx^n$. Likewise, the zero-section for the nonequivariant bundle $\upzeta: (\mr{B}\Sigma_2)_+ \to \mr{Th}(\upgamma^{\oplus n})$ sends $\uu_n \mapsto \ttt^n.$ By naturality of $ \modPhi$ and $\uplambda$, we get a commutative diagram  
\[ 
\begin{tikzcd}
\rHC^{\star}\!(\mr{Th}(\overline{\upgamma}^{\oplus n})) \dar["\upzeta_{\Ctwo}^{*}"] \rar["\modPhi"] & {\rH^{*}(\Fix{\mr{Th}(\overline{\upgamma}^{\oplus n})})} \rar["\uplambda_{\mr{S}^{\uprho \otimes \upomega}}^*"] \dar["{(\Fix{\upzeta}_{\Ctwo})^{*}}"] & \rH^{*}(\mr{Th}(\upgamma^{\oplus n})) \dar["\upzeta^*"] \\
\rHC^{\star}\!((\mr{B}_{\Ctwo}\Sigma_2)_+) \rar["\modPhi"'] & \rH^{*}(\Fix{\mr{B}_{\Ctwo}\Sigma_2}_+) \rar["\uplambda_{\mr{S}^0}^*"'] & \rH^{*}((\mr{B}\Sigma_2)_+).
\end{tikzcd}
\]
which along with \autoref{rmk:Phiproj} and injectivity of $\upzeta^*$ implies the result.  
\end{proof}

\begin{cor} For any space $\rX \in \Top_*^{\Ctwo}$ and a class $ u \in \rHC^{n \uprho}(\rX)$, 
\begin{equation} \label{Powercompare}
\mathcal{P}_n(\modPhi(u)) = \uplambda^*_{\rX}(\modPhi(\mathcal{P}_{n \uprho}(u))).
\end{equation}
\end{cor}

\begin{proof}
It is enough to show that in the following diagram  commutes as the  blue path and the red path indicates the left-hand side and the right-hand side of \eqref{Powercompare} respectively. 
\[ 
\begin{tikzpicture}
\node[scale=0.9] at (0,0) {
\begin{tikzcd}[row sep=8ex, column sep=8ex]
& \mr{D}_2(\CFix{\rX})  \dar[color = blue, "\mr{D}_2(\GFix{u})"']  \ar[rd, phantom, "\text{(A)}"]  \rar[color = red, "\uplambda_{\rX}"] & \CFix{\mr{D}_2^{\Ctwo}(\rX)}\dar[color= red, "\GFix{\mr{D}_2^{\Ctwo}(u)}"] \\
\mr{D}_2(\Sigma^n \rH\F)\dar[color = blue, "\updelta_2"'] \ar[rd, phantom, "\text{(B)}"]  & \lar[color = blue, "\mr{D}_2(\Sigma^n \uppi^{(0)}_{\F})"']  \mr{D}_2(\GFix{\Sigma^{n \uprho} \rH\uF}) \ar[rd, phantom, "\text{(C)}"]  \rar["\uplambda^{\Phi}_{\Sigma^{n \uprho} \sma \rH\uF}"]  \dar["\updelta_2"]  & \GFix{\mr{D}_2^{\Ctwo}(\Sigma^{n \uprho} \rH\uF )}  \dar[color = red, "\Phi(\updelta_2^{\Ctwo})" ]  \\
\mr{D}_2(\mr{S}^n) \sma \mr{D}_2(\rH\F) \dar[color = blue, "\uu_n \sma \Theta_2^{\F}"'] \ar[rd, phantom, "\text{(E)}"]  & \lar["\mathbbm{1} \sma\mr{D}_2 (\uppi^{(0)}_{\F}) "'] \mr{D}_2(\mr{S}^n) \sma \mr{D}_2(\GFix{\rH\uF})  \dar["\uu_n \sma \Theta_2^{\Phi \F}"'] \rar \ar[rd, phantom, "\text{(F)}"] & \GFix{\mr{D}_2^\Ctwo(\mr{S}^{n \uprho}) }\sma\GFix{ \mr{D}_2^{\Ctwo}(\rH\uF)} \dar[color = red, "\underline{\uu}_n \sma \Theta_2^{ \uF}"]   \\
\Sigma^{2n} \rH\F \sma \rH\F \dar[color = blue, "\Sigma^{2n}\upmu_{\F} "'] \ar[rrd, phantom, "\text{(G)}"]  & \Sigma^{2n} \rH\F \sma \GFix{\rH\uF} \lar["\mathbbm{1} \sma \uppi_{\F}^{(0)}"]  & \lar["\Sigma^{2n}\uppi^{(0)}_{\F} \sma \mathbbm{1}"]  \GFix{\Sigma^{2n \uprho}\rH\uF} \sma \GFix{ \rH \uF} \dar[color = red, "\Sigma^{2n} \GFix{\upmu_{\uF}}"] \\
\Sigma^{2n}\rH\F && \Sigma^{2n}\GFix{\rH\uF} \ar[ll, color = red, "\Sigma^{2n} \uppi^{(0)}_{\uF}"]
\end{tikzcd}
}; \end{tikzpicture}
\]
The squares (A), (B) and (C) commute naturally, the squares (E) and (G) commute because $\pi_{\F}^{(0)}$ is an $\bE_\infty$-ring map, and (F) commutes because of \autoref{lem:eulercompare}.
\end{proof}

\begin{proof}[Proof of \autoref{thm:squarefix}]
For any space $\rX \in \Top_*^{\Ctwo}$ and a class $ u \in \rHC^{n \uprho}(\rX)$, we have a commutative diagram 
\[ 
\begin{tikzcd}
(\mr{B}\Sigma_2)_+ \sma \Fix{\rX} \ar[rr, "\Delta_{\Fix{\rX}}"] \dar["\uplambda_{\mr{S}^0} \sma \mathbbm{1}_{\Fix{\rX}}"']  && \mr{D}_2(\Fix{\rX})) \dar["\uplambda_{\rX}"]\\
(\mr{B}_{\Ctwo}\Sigma_2)_+ \sma \Fix{\rX} \ar[rr, "\Fix{(\Delta_{\rX}^{\Ctwo})}"'] && \Fix{\mr{D}_2^{\Ctwo}(\rX)} .
\end{tikzcd}
\]
Therefore, 
\begin{eqnarray*}
\sum_{i= 0 }^n \ttt^{n-i} \otimes \Sq^i(\Phi_*(u)) &=& \Delta_{\Fix{\rX}}^*(\mathcal{P}_n(\Phi_*(u)) )   \\
&=&  \Delta_{\Fix{\rX}}^*(\uplambda^*_{\rX}(\Phi_*(\mathcal{P}_{2n,n}(u)))) \\
&=&  (\uplambda_{\mr{S}^0}^* \otimes \mathbbm{1}_{\Fix{\rX}}^*) ((\Fix{(\Delta_{\rX}^{\Ctwo})})^*(\Phi_*(\mathcal{P}_{2n,n}(u))))  \\
 &=& (\uplambda_{\mr{S}^0}^* \otimes \mathbbm{1}_{\Fix{\rX}}^*) \Phi_*((\Delta^\Ctwo_\rX)^*(\mathcal{P}_{2n,n}(u)) )\\
&=& (\uplambda_{\mr{S}^0}^* \otimes \mathbbm{1}_{\Fix{\rX}}^*) \Phi_* \left(\sum_{i = 0}^{n} \xx^{n-i} \otimes \Sq^{2i,i}(u) \right. \\
&& \hspace{80pt} \left.+ \sum_{i = 0}^{n}\yy \xx^{n-i-1} \otimes \eSq^{2i+1}(u) \right)
\end{eqnarray*}
\begin{eqnarray*}
\hspace{112pt} &=& (\uplambda_{\mr{S}^0}^* \otimes \mathbbm{1}_{\Fix{\rX}}^*) \left(\sum_{i = 0}^{n} \ttt^{n-i} \otimes \Phi_*(\eSq^{2i}(u)) \right. \\
&& \hspace{80pt} \left.+ \sum_{i = 0}^{n}\iota\ttt^{n-i-1} \otimes \Phi_*(\eSq^{2i+1}(u) )\right)\\
&=& \sum_{i = 0}^{n} \ttt^{n-i} \otimes \Phi_*(\eSq^{2i}(u)),
\end{eqnarray*}
and hence, the result is true for all $u \in \rHC^{n\uprho}(\rX)$ for any  $\rX \in \Top_*^{\Ctwo}$. 

Since the squaring operations are stable, the result extends to arbitrary $\RO(\Ctwo)$-graded cohomology classes. Moreover, since
the geometric fixed point functor $\Phi$ commutes with suspensions \eqref{GFixsuspension},
and any $\mr{E}\in \SpfinC$ is equivalent to $\Sigma^{-n} \Sigma^\infty_\Ctwo \rX$ for some $n$ and $\rX\in \Top_*^\Ctwo$, we conclude the same for any $u \in \rHC^{\star}(\mr{E})$. 
\end{proof}

\section{ Topological realization of $\cAR(1)$}
\label{sec:Trealization}

We begin by proving \autoref{list},  which identifies all possible $\cAR$-module structures on $\cAR(1)$ up to isomorphism. 

\begin{figure}[h]
\includegraphics[width=0.6\textwidth]{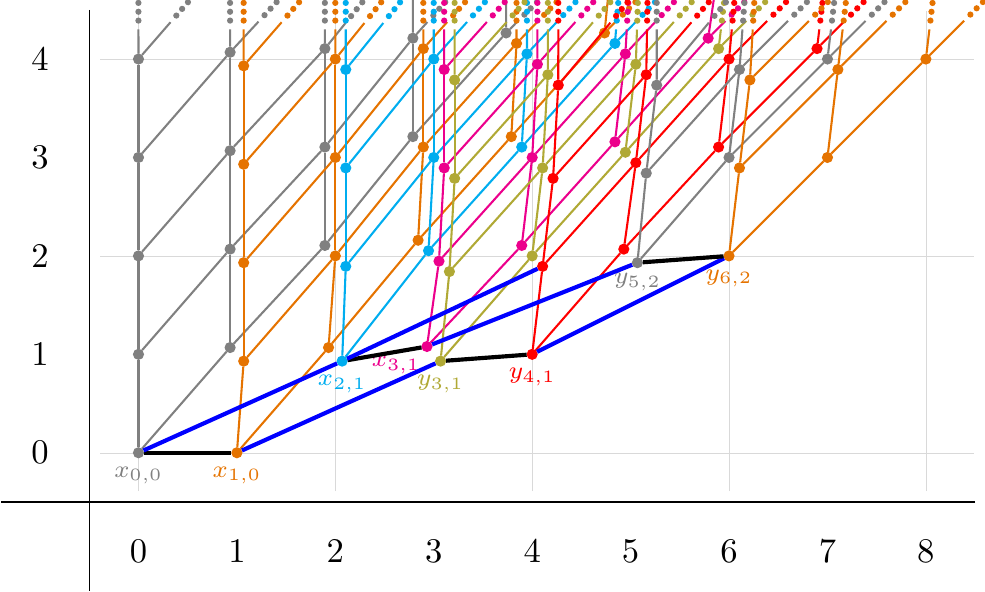}
\caption{This figure displays the free $\bMR$-module $\cAR(1)$, with $\Sq^1$ and $\Sq^2$-multiplications drawn in only on the $\bMR$-module generators. 
}
\label{fig:motivicA1Roots}
 \end{figure}
\begin{proof}[Proof of \autoref{list}]  Note that the Cartan formula of $\cAR$  and finiteness of $\cAR(1)$ imply that the $\cAR$-module structure on $\cAR(1)$ is determined once  the action of $\Sq^4$ and $\Sq^{8}$ are specified on its $\bM^\R$-generators. 
The following are possible $\Sq^4$ and $\Sq^8$-actions on the $\bMR$-module generators. As can be seen in \autoref{fig:motivicA1Roots}, there is no room for other possible actions.
\begin{eqnarray*}
\Sq^4(x_{0,0}) &=& \beta_{03}  (\rho \cdot y_{3,1}) + \beta_{04} ( \tau \cdot y_{4,1}) + \alpha_{03}(\rho \cdot x_{3,1} )\\
\Sq^4(x_{1,0}) &=& \beta_{14} (\rho \cdot y_{4,1}) + \beta_{15}(y_{5,2})  \\
\Sq^4(x_{2,1})&=&   \jay_{24} (\rho^2 \cdot y_{4,1}) + \beta_{25} (\rho \cdot y_{5,2}) +\beta_{26} (\tau \cdot y_{6,2})\\
\Sq^4(x_{3,1})&=& \beta_{36} (\rho \cdot y_{6,2})\\
\Sq^4(y_{3,1})&=& \gamma_{36} (\rho \cdot y_{6,2}) \\
\Sq^8(x_{0,0}) &=& \beta_{06} (\rho^2 \cdot y_{6,2})
\end{eqnarray*}
The Adem relation $\Sq^2 \Sq^3 = \Sq^5 + \Sq^4 \Sq^1 + \rho \Sq^3 \Sq^1$ (see \autoref{RMotivicAdem}), when applied to $x_{0,0}$ and $x_{2,1}$, yields $\beta_{15} =1$, $\beta_{03} + \beta_{04} + \beta_{14}  =1$ and $\beta_{25} + \beta_{26}= \beta_{36}$.  The equation 
\[\jay_{24}= \beta_{03}  \gamma_{36} + \alpha_{03}\beta_{36},\] 
is forced  by the Adem relation $ \Sq^4 \Sq^4 = \Sq^2 \Sq^4 \Sq^2 + \tau \Sq^3 \Sq^4 \Sq^1$ when 
applied to $x_{0,0}$. This exhausts all constraints imposed by Adem relations in these dimensions.
\end{proof}  

In \autoref{list}, there are exactly seven free variables taking values in $\F$, and therefore, there are exactly 128 different $\cAR$-module structure on $\cAR(1)$. Thus, in order to complete the proof of \autoref{thm:128R}, we  realize these $\cAR$-modules as spectra using \autoref{wTRT3}, which is a weak form of the $\R$-motivic Toda realization theorem. 

 \begin{proof}[Proof of \autoref{thm:128R}]
 Firstly, note that $\cAR_{ \overline{\mr{v}}}(1)$ is a cyclic $\cA^{\R}$-module for all $\overline{\mr{v}} \in \mathcal{V}$, therefore $\cA^\C_{ \overline{\mr{v}}}(1) := \cAR_{ \overline{\mr{v}}}(1)/(\rho) $ admits a May filtration. Secondly, note that 
\[ \gr(\cA^\C_{ \overline{\mr{v}}}(1)) \cong \Lambda_{\bM^\C}(\upxi_{1,0}, \upxi_{1,1}, \upxi_{2,0}) \]
 as an $\gr(\cA^\C)$-module (see \eqref{xi-ij} for notation). Consequently, 
 \begin{equation} \label{MayA1}
   {^{\mr{May}}}\mr{E}_{1, \cA^\C_{ \overline{\mr{v}}}(1) }^{*,*,*,*} \cong  {^{\mr{May}}}\mr{E}_{1, \bM^\C }^{*,*,*,*} /(\sfh_{1,0}, \sfh_{1,1}, \sfh_{2,0}) \cong \frac{\bM^\C[\sfh_{i,j}: i \geq 1, j \geq 0] }{(\sfh_{1,0}, \sfh_{1,1}, \sfh_{2,0})}. 
   \end{equation}
In the notation of \autoref{subsec:wTRT3} 
\[ \mathcal{D}_{\cAR(1)} = \{ (0,0), (1,0),(2,1), (3,1), (4,1),(5,2), (6,2) \}\] 
By directly inspecting the $(s,f,w)$-degree of   ${^{\mr{May}}}\mr{E}_{1, \cA_{ \overline{\mr{v}}}^\C(1) }^{*,*,*,*}$, 
we see that the condition necessary for existence in  \autoref{wTRT3} is satisfied. Hence, the result. 
 \end{proof}
\begin{rmk} \label{nonuniqueA}The vanishing region of ${^{\mr{May}}}\mr{E}_{1, \cA_{ \overline{\mr{v}}}^\C(1) }^{*,*,*,*}$ does not preclude the possibility of having  a nonzero element in $\Ext_{\cAR}^{-1, 2, 0}(\mr{M}, \M)$. We suspect (even after running  the differentials in \eqref{aAHSS} and \eqref{rhoBSS}), that the above group is nonzero for a given $\cAR$-module structure on $\cAR(1)$,  and that there are, up to homotopy,   multiple realizations as  $\R$-motivic spectra. 
\end{rmk}

Our next goal is to prove \autoref{thm:exactcofiber}. We begin with the following observation. 

\begin{lem} \label{uniqueB}The $\cAR$-modules  $\mathcal{B}_{2}^\R(1)$ and $\mathcal{B}_{\sfh}^\R( 1)$ are uniquely realizable as objects in $\SpfinR$.  
\end{lem}

\begin{proof} Both $\mathcal{B}_{2}^\R(1)$ and $\mathcal{B}_{\sfh}^\R(1)$ are cyclic $\bMR$-free finite $\cAR$-module and  
\[ \mathcal{B}_{2}^\R(1)/(\rho) \cong \mathcal{B}_{\sfh}^\R(1)/(\rho)\] 
as $\cA^\C$-module. Let $\mathcal{B}^{\C}(1) := \mathcal{B}_{2}^\R(1)/(\rho)$. 
It is easy to see that $\gr(\mathcal{B}^{\C}(1))$ is isomorphic to  $\Lambda_{\bM^\C}(\upxi_{1,0}, \upxi_{1,1})$ 
as an $\gr(\cA^\C)$-module, and therefore
 \begin{equation} \label{mayB}
   {^{\mr{May}}}\mr{E}_{1, \mathcal{B}^{\C}(1)}^{*,*,*,*}  \cong     \F[\tau][ h_{i,j}: i \geq 1, j \geq 0]/(\sfh_{1,0}, \sfh_{1,1}) .
 \end{equation}
Using this, along with the fact that
\[ \mathcal{D}_{\mathcal{B}^{\C}(1)} = \{ (0,0), (1,0),(2,1), (3,1), \}\] 
 shows that the condition necessary for existence as well as uniqueness in  \autoref{wTRT3} is satisfied. Hence, the result. 
\end{proof}

\begin{proof}[Proof of \autoref{thm:exactcofiber}]
 Consider the injective $\cA^\R$-module  map $\Sigma^{3,1}\mathcal{B}_{\epsilon}^\R(1) \to \mr{A}^\R_{\overline{\mr{v}}}$ sending the $\bM^\R$ generator in degree $(3,1)$ to $x_{3,1}  + y_{3,1}$. It follows from  \autoref{list} that the quotient is isomorphic to {$\mathcal{B}_{\delta}^\R(1)$}. Thus, we have the exact sequence  \eqref{exactBAB}. 
 
The topological realization of \eqref{exactBAB}, i.e. the cofiber sequence \eqref{cofiberYA1}, would follow  immediately from \autoref{uniqueB} once we show that any one of the $\cAR$-module maps in  \eqref{exactBAB} can be realized as a map in $\SpfinR$.  Thus, it is enough to show that the nonzero class in the $\mr{E}_2$-page represented by the projection map $\cAR_{\overline{\mr{v}}}(1) \twoheadrightarrow \mathcal{B}_{\delta}^\R(1)$ in degree $(0,0,0)$ of  the $\R$-motivic Adams spectral sequence 
\begin{equation} \label{assBA}
 \mr{E}_2^{s,f,w} := \Ext^{s, f, w}_{\cAR}(\cAR_{\overline{\mr{v}}}(1) , \mathcal{B}_{\delta}^\R(1)) \Rightarrow \left[ \Y_{(\delta,1)}^\R,  \cA_{1}^\R [\overline{\mr{v}}] \right]_{s,w}
 \end{equation}
  is a nonzero permanent cycle. 
  
  Using \eqref{MayA1}, the $\rho$-Bockstein  spectral sequence \eqref{rhoBSS} for $\cA_1^\R[\overline{\mr{v}}]$ and the Atiyah-Hirzebruch spectral sequence
  \[  \mathcal{B}^{s', w'}_{\mathcal{B}_{\delta}^\R(1)} \otimes  \Ext^{s, f, w}_{\cAR}(\cAR_{1}[\overline{\mr{v}}] , \bMR) \Rightarrow  \Ext^{s+s', f, w+w'}_{\cAR}(\cAR_{1}[\overline{\mr{v}}] , \mathcal{B}_{\delta}^\R(1)), \]
  one can easily check  $\mr{E}_2^{-1,f,0} = 0$ for all  $f \geq 2$, and thus any nonzero element in degree $(0,0,0)$ of the $\mr{E}_2$-page in \eqref{assBA} is, in fact, a nonzero permanent cycle.
\end{proof}

Our next goal is to analyze the underlying spectrum and  geometric fixed-points spectrum of $\cA_1^\Ctwo[\overline{\mr{v}}]$, the Betti realization of  $\cA_1^\R[\overline{\mr{v}}]$. 

\subsection{The Betti realization of $\cAR_1$}
\

 Under the Betti-realization map 
\begin{equation} \label{betipoint}
\begin{tikzcd}
\upbeta_*: \pi_{\ast, \ast}\RHF \cong \F[\rho, \tau]\rar &  \pi_{\star} \rH\uF
\end{tikzcd}
\end{equation}
$\rho \mapsto a_{\upsigma}$ and $\tau \mapsto u_{\upsigma}$. Since the functor $\upbeta$ is symmetric monoidal and $\upbeta(\rH_\R\F)= \rH \uF$, the $i$-th $\R$-motivic squaring operations maps to the $i$-th $\RO(\Ctwo)$-graded squaring operations under the map 
\[ 
\begin{tikzcd}
 \upbeta_*: \cAR \rar & \cA^\Ctwo. 
 \end{tikzcd}
\]
Hence,  $\rHC^{\star}\!(\cA_1^\Ctwo[\overline{\mr{v}}])$ is $\bMC$-free (as $\rHR^{*,*}(\cA_1^\R[\overline{\mr{v}}])$ is $\bM^\R$-free) and its $\cA^\Ctwo$-module structure is essentially given by  \autoref{list} (after replacing $\Sq^i$ with $\eSq^i$ and $\bM^\R$-basis elements by its image under $\upbeta_*$).

\begin{rmk} \label{nochange} The map  $\upbeta_*$ of \eqref{betipoint} is only an injection with cokernel the summand  $\Theta\{ u_{\upsigma}^{-i} a_{\upsigma}^{-j}: i,j \geq 0\}$ of $\bMC$.  In general, for an $\cA^\R(1)$-module $\mr{M}_\R$,  the number of $\cA^\Ctwo$-module structures on $\upbeta(\mr{M}_\R)$ can be strictly larger  than the number of $\cA^\R$-module structures on $\M_\R$. But this is not the case when $\M_\R = \cA^\R_{\overline{\mr{v}}}(1)$ simply for degree reasons, therefore \autoref{cor:128C} holds.
\end{rmk}

 As discussed in \autoref{eg:ResHF},
 the restriction map
 \[
 \begin{tikzcd}
 \Res_*: \bMC \rar &  \F
 \end{tikzcd}
 \] 
 sends $a_{\upsigma} \mapsto 0$,  $u_{\upsigma} \mapsto 1$, and $\Theta\mapsto 0$. 
 Thus, when  $\rHC^{\star}\!(\mr{E})$ is $\bMC$-free, $\Res_*$  is simply ``setting $a_{\upsigma} = 0$,  $u_{\upsigma} =1$, and $\Theta=0$".  
This observation, along with \autoref{thm:squareunder}, allows us to completely deduce  the $\cA$-module structure of $\rH^*(\ParRes{ \cAR_{\overline{\mr{v}}}(1)})$  from \autoref{list}. 
Together with the fact that the $\cA$-module structures on $\cA(1)$ are uniquely-realized, our observations yield 
the following theorem, where the notation  {$\mr{A}_1[i,j]$} is adopted from \cite{BEM}. 

\begin{thm} \label{underA1} 
For $ \overline{\mr{v}} = ( \alpha_{03}, \beta_{03},\beta_{14}, \beta_{06}, \beta_{25}, \beta_{26}, \gamma_{36}) \in \mathcal{V}$ (as in \autoref{list}), \[ \ParRes{\cA_1^\Ctwo[\overline{\mr{v}}]} \simeq \mr{A}_1[1+ \beta_{03} + \beta_{14}, \beta_{26}].\] 
\end{thm}

Now we shift our attention towards understand the geometric fixed-points of $\cA_1^\Ctwo[\overline{\mr{v}}]$. 
As discussed in \autoref{eg:FixHF}, the modified geometric fixed-points functor 
\[ 
\begin{tikzcd}
\modPhi \colon \bMC \rar & \F 
\end{tikzcd}
\]
sends $a_{\upsigma} \mapsto 1$,  $u_{\upsigma} \mapsto 0$, and $\Theta\mapsto 0$. Thus, when  $\rHC^{\star}\!(\mr{E})$ is $\bMC$-free, $\modPhi$  is simply ``setting $a_{\upsigma} =1$,  $u_{\upsigma} =0$, and $\Theta=0$".  This, along with  \autoref{list} and  \autoref{thm:squarefix}, gives the following.
\begin{notn}
Because $\rHC^{\star}\!(\cA_{1}^\Ctwo [\overline{\mr{v}}])$ is $\bMC$-free, the  $\rH\F$-cohomology of  $\GFix{\cA_{1}^\Ctwo [\overline{\mr{v}}]}$ consists of eight $\F$-generators, all of which are in the image of $\modPhi$. We let 
\[ \text{$s_{0} := \modPhi(x_{0,0})$, $s_{1a} := \modPhi(x_{2,1})$, $s_{1b} := \modPhi(x_{1,0})$, $s_{2} := \modPhi(y_{3,1})$}\]
\[ \text{$t_{2} := \modPhi(x_{3,1})$, $t_{3a} := \modPhi(y_{5,2})$, $t_{3b} := \modPhi(y_{4,1})$, $t_{4} := \modPhi(y_{6,2})$}.\]
Note that  $|s_{i (-)}| = |t_{i(-)}| = i$. 
\end{notn}

\begin{thm} \label{GFixA1} Let $ \overline{\mr{v}} = ( \alpha_{03}, \beta_{03},\beta_{14}, \beta_{06}, \beta_{25}, \beta_{26}, \gamma_{36}) \in \mathcal{V}$, and let 
\[ \jay_{24}= \beta_{03} \gamma_{36} + \alpha_{03}(\beta_{25} + \beta_{26}) \]
as in \autoref{list}.  The $\cA$-module structure on $\rH^*(\GFix{\cA_{1}^\Ctwo [\overline{\mr{v}}]})$  is determined by the following relations, as depicted in \autoref{fig:HPhiA1}:
\begin{multicols}{2}
\begin{itemize}
\item $\Sq^1(s_0) = s_{1a}$ \vspace{2pt}
\item $\Sq^1(s_{1b}) = s_2$ \vspace{2pt}
\item $\Sq^1(t_2) = t_{3a}$ \vspace{2pt}
\item $\Sq^1(t_{3b}) = t_4$\vspace{2pt}
\item $\Sq^2(s_0) =  \beta_{03} s_2 + \alpha_{03} t_2$ 
\end{itemize}
\begin{itemize}
\item $\Sq^2(s_{1a}) = \beta_{25} t_{3a}  + \jay_{24} t_{3b}$ \vspace{2pt}
\item $\Sq^2(s_{1b} ) = t_{3a} + \beta_{14} t_{3b}$ \vspace{2pt}
\item $\Sq^2(s_2) = \gamma_{36}t_4 $ \vspace{2pt} 
\item $\Sq^2(t_2) =(\beta_{25} + \beta_{26})t_4$ \vspace{2pt}
\item $ \Sq^4 (s_0) =\beta_{06} t_4$. 
\end{itemize}
\end{multicols}

\end{thm}

\begin{figure}[h]
\[
\begin{tikzpicture}\begin{scope}[ thick, every node/.style={sloped,allow upside down}, scale=0.6, font=\tiny]
\draw (0,0)  node[inner sep=0] (s0) {} -- (0,1) node[inner sep=0] (s1a) {};
\draw (1.5,2)  node[inner sep=0] (t2) {} -- (1.5,3) node[inner sep=0] (t3a) {};
\draw (3,1)  node[inner sep=0] (s1b) {} -- (3,2) node[inner sep=0] (s2) {};
\draw (4.5,3)  node[inner sep=0] (t3b) {} -- (4.5,4) node[inner sep=0] (t4) {};
 \draw [color=blue,bend right] (s0) to node[pos=0.5,above] {$\alpha_{03}$} (t2);
  \draw [color=blue,out=20,in=-160] (s0) to node[pos=0.5,below] {$\beta_{03}$} (s2);  
  \draw [color=blue,bend left=20] (s1a) to node[pos=0.3,below={-1pt}] {$\beta_{25}$}  (t3a);
  \draw [color=blue,bend left=90] (s1a) to node[pos=0.5,above] {$\jay_{24}$} (t3b);
   \draw [color=blue,out=150,in=-30] (s1b) to (t3a);
  \draw [color=blue,bend right] (s1b) to  node[pos=0.5,below] {$\beta_{14}$}(t3b);
   \draw [color=blue,out=30,in=210] (s2) to node[pos=0.5,below] {$\gamma_{36}$} (t4);
   \draw [color=blue,out=15,in=195] (t2) to node[pos=0.7,above] {$\beta_{25}+\beta_{26}$} (t4);
\draw [color = red] (s0) to (5.5,0) to (5.5,4) to (t4);
\filldraw (s0) circle (2.5pt);
\filldraw (s1a) circle (2.5pt);
\filldraw (t2) circle (2.5pt);
\filldraw (t3a) circle (2.5pt);
\filldraw (s1b) circle (2.5pt);
\filldraw (s2) circle (2.5pt);
\filldraw (t3b) circle (2.5pt);
\filldraw (t4) circle (2.5pt);
\draw (s0) node[left]{$ s_0$} (0,1) node[left]{$ s_{1a}$} (s1b) node[below]{$s_{1b} $} (s2) node[right]{$s_2 $};
\draw (t2) node[left]{$ t_2$} (t3a) node[right]{$ t_{3a}$} (t3b) node[below right={-2pt}]{$t_{3b} $} (t4) node[above]{$t_4 $};
\draw (5.5,2) node[right]{$\beta_{06}$};
\end{scope}\end{tikzpicture}
\]
\caption{The $\cA$-module  $\rH^*(\GFix{\cA_{1}^\Ctwo [\overline{\mr{v}}]})$ }
\label{fig:HPhiA1}
\end{figure}

\section{An $\R$-motivic analogue of the spectrum $\mathcal{Z}$ }
\label{sec:Z}

The type $2$ spectrum $\mathcal{Z} \in \Sp_{2, \mr{fin}}$, introduced in \cite{BE20}, is defined by the property that its cohomology as an $\cA(2)$-module   is
\[ \mathcal{B}(2) := \cA(2) \otimes_{\Lambda(\mr{Q}_2)} \F, \]
where $\mr{Q}_2 = [\Sq^4, \mr{Q}_1]$ is dual to the Milnor generator $\upxi_3$ of the dual Steenrod algebra.  They first show that an $\cA$-module structure on $\cA(2)$ satisfying the criteria in \cite{BE20}*{Lemma~2.7} leads to an $\cA$-module structure on $\mathcal{B}(2)$.   In \cite{BE20}, the authors show that among the 1600 possible $\cA$-module structures on $\cA(2)$ \cite{Roth}, there are some $\cA$-modules that satisfy \cite{BE20}*{Lemma~2.7}. Then they use the classical Toda realization theorem to show that any $\cA$-module whose underlying $\cA(2)$-module structure is $\mathcal{B}(2)$ can be realized as a $2$-local finite spectrum, which they call $\mathcal{Z}$.  

We construct $\mathcal{Z}_\R \in \SpfinR$ by emulating the construction of the classical $\mathcal{Z}$ (as in \cite{BE20}) in the $\R$-motivic context. Since there is no a priori $\cAR$-module structure on $\cAR(2)$,
we produce one in the following subsection. In fact, we construct an $\R$-motivic spectrum whose cohomology is the desired  $\cAR$-module.  

\subsection{A topological realization of $\cAR(2)$}

 Let $\cAR(2)$ denote the sub-$\bMR$-algebra of the $\R$-motivic Steenrod algebra $\cAR$ generated  by $\Sq^1, \Sq^2$, and $ \Sq^4$. We will use a method of Smith (exposed in \cite{Rav}*{Appendix C}) to construct an $\R$-motivic spectrum $\cAR_2 \in \SpfinR$ such that its cohomology as an $\cAR(2)$-module is free on one generator. 
 
Let $\sfh, \eta_{1,1}$ and $\nu_{3,2}$ denote the first three $\R$-motivic Hopf-elements.  

\begin{lem} \label{TodaBHopf}The $\R$-motivic Toda-bracket $\langle \sfh, \eta_{1,1}, \nu_{3, 2} \rangle$ contains $0$. 
 \end{lem}
\begin{proof} 
In this argument, it will be convenient to refer to the ``coweight'', by which we mean the difference $s-w$, as in \cite{C2MW0}. 

Since $\sfh$ and $\eta_{1,1}$ have coweight $0$ while $\nu_{3,2}$ has coweight $1$, it follows that the bracket $\langle \sfh, \eta_{1,1}, \nu_{3, 2} \rangle$ is comprised of elements in stem $5$ with coweight $2$. The only element in  stem 5 with coweight 1 is  $\rho \cdot \nu_{3,2}^2$ \cite{BI}. Since this element is a $\nu_{3,2}$ multiple, it lies in the indeterminacy, which means that the $\R$-motivic Toda-bracket does  contain zero.
\end{proof}

\autoref{TodaBHopf} implies that we can construct a $4$-cell complex $\mathcal{K}$ whose cohomology as an {$\cA^\R$}-module has the  structure described in \autoref{Kexists} and displayed in \autoref{fig:HK}. 

\begin{cor}\label{Kexists} There exists $\mathcal{K} \in \SpfinR$ such that $\rHR^{*,*}(\mathcal{K})$ is $\bMR$-free  on four generators $x_0, x_1,x_3$ and $x_7$, such that $\Sq^{i+1}(x_i) = x_{2i+1}$ for $i \in \{ 0,1,3 \}$. 
\end{cor} 

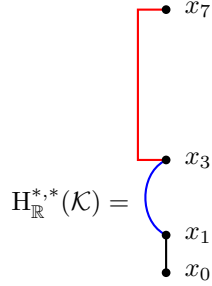
\begin{figure}[h]
\[
\rHR^{*,*}(\mathcal{K})  = \!\! \raisebox{-3em}{
\begin{tikzpicture}\begin{scope}[ thick, every node/.style={sloped,allow upside down}, scale=0.5]
\draw (0,0)  node[inner sep=0] (v0) {} -- (0,1) node[inner sep=0] (v1) {};
\draw (0,3) node[inner sep=0] (v3) {};
\draw (0,7) node[inner sep=0] (v7) {};
 \draw[blue] (v1) to [out=150,in=-150] (v3);
 \draw [color=red] (0,3) -- (-.75,3) -- (-.75,7) -- (0,7);
\filldraw (v0) circle (2.5pt);
\filldraw (v1) circle (2.5pt);
\filldraw (v3) circle (2.5pt);
\filldraw (v7) circle (2.5pt);
\draw (0,0) node[right]{$\ x_0$};
\draw (0,1) node[right]{$\ x_1$};
\draw (0,3) node[right]{$\ x_3$};
\draw (0,7) node[right]{$\ x_7$};
\end{scope}\end{tikzpicture}}
\]
\caption{We depict the $\cAR$-structure of $\rHR^{*,*}(\mathcal{K})$ by marking the $\Sq^1$-action by black straight lines,  the $\Sq^2$-action by blue curved lines, and the $\Sq^4$-action by red lines between the $\bMR$-generators. }
\label{fig:HK}
\end{figure}

Let $e \in \Z_{(2)}[\Sigma_6]$ denote the idempotent corresponding to the Young tableaux
\[ 
\begin{tikzpicture}[scale =.5]
\draw  (0,0) -- (0,3) -- (3,3) -- (3,2) -- (2,2) -- (2,1) -- (1,1) -- (1,0) -- (0,0);
\draw (0,1) -- (1,1) -- (1,2) -- (2,2) -- (2,3);
\draw (0,2) -- (1,2) -- (1,3);
\node[] at (.5,.5) {6};
\node[] at (.5,1.5) {4};
\node[] at (.5,2.5) {1};
\node[] at (1.5,1.5) {5};
\node[] at (2.5,2.5) {3};
\node[] at (1.5,2.5) {2};
\end{tikzpicture}
\]
which is constructed as follows. Let $\Sigma_{\mr{Row}} \subset \Sigma_6$ denote the subgroup  comprised of permutations that preserve each row. Likewise, let $\Sigma_{\mr{Col}}$ denote the subgroup comprised of column-preserving permutations. 
Let 
\begin{equation} \label{RCFormula}
\sfR = \sum_{r \in \Sigma_{\mr{Row}}} r  \qquad  \text{and} \qquad  \sfC = \sum\limits_{c \in \Sigma_{\mr{Col}}} (-1)^{\mr{sign}(c)}c 
\end{equation}
   and define 
\[ e = \frac{1}{\mu} \sfR \cdot \sfC,\]
where $\mu$ is an odd integer  defined in \cite[Theorem C.1.3]{Rav}. We let $\overline{e}$ denote the resulting idempotent in $\F[\Sigma_6]$.

\begin{prop}\label{e-dimcount}
The idempotent {$\overline{e} \in \F[\Sigma_6]$} has the property that {$\overline{e} (\mr{V}^{\otimes 6}) = 0 $} if $\dim_{\F} \mr{V} < 3$ and  
\[
  \dim_{\F} \overline{e}(\mr{V}^{\otimes 6})  = \left\lbrace  
  \begin{array}{cccc}
  8 & \text{ if $\dim_{\F} \mr{V} = 3$}  \\
  64 & \text{ if $\dim_{\F} \mr{V} = 4$} 
  \end{array}
  \right. 
\]
  \end{prop}
  
  \begin{pf}
 Let $\overline{\sfR}$ and $\overline{\sfC}$ denote the images of $\sfR$ and $\sfC$ in $\F[\Sigma_6]$, respectively. Then $\overline{e} = \overline{\sfR} \cdot \overline{\sfC}$. It is straightforward that $\overline{\sfC}$ vanishes on $\mr{V}^{\otimes 6}$ if $\dim \mr{V} \leq 2$. 
 
 Now suppose that $\mr{V}$ has basis $\{a,b,c\}$. Then a basis for $\overline{e}( \mr{V}^{\otimes 6})$ is given by
 \[
\left\{
\overline{e} \left( \hspace{-1ex}\raisebox{-2ex}{\Young{a}{b}{c}{b}{c}{c} } \hspace{-2ex}\right), 
\overline{e} \left( \hspace{-1ex}\raisebox{-2ex}{\Young{b}{a}{c}{a}{c}{c} } \hspace{-2ex}\right), 
\overline{e} \left( \hspace{-1ex}\raisebox{-2ex}{\Young acbcbb } \hspace{-2ex}\right), 
\overline{e} \left( \hspace{-1ex}\raisebox{-2ex}{\Young cababb } \hspace{-2ex}\right), 
\overline{e} \left( \hspace{-1ex}\raisebox{-2ex}{\Young cbabaa } \hspace{-2ex}\right), 
\overline{e} \left( \hspace{-1ex}\raisebox{-2ex}{\Young bcacaa } \hspace{-2ex}\right), 
 \right.
 \]
 \[
 \left.
\overline{e} \left( \hspace{-1ex}\raisebox{-2ex}{\Young abcbac } \hspace{-2ex}\right), 
\overline{e} \left( \hspace{-1ex}\raisebox{-2ex}{\Young acbcab } \hspace{-2ex}\right)
 \right\}.
 \]
 
 Finally, suppose that $\dim \mr{V} = 4$ with basis $\{a,b,c,d\}$. For any subspace $\mr{W} \subset \mr{V}$ spanned by three of these basis elements, the space $\overline{e}( \mr{W}^{\otimes 6})$ has dimension 8, as we have just seen. There are $4$  choices of $\mr{W}$, which together yield a 32-dimensional subspace of $\mr{V}^{\otimes 6}$. Now consider Young tableaux in which all four basis elements appear and only one is repeated. In the case that $d$ is repeated, we generate only two independent elements:
 \[
\overline{e} \left( \hspace{-1ex}\raisebox{-2ex}{\Young acdbdd } \hspace{-2ex}\right) \quad \text{and} \quad 
\overline{e} \left( \hspace{-1ex}\raisebox{-2ex}{\Young cadbdd } \hspace{-2ex}\right).
 \]
 Allowing any basis element to be the repeating one, this gives an 8-dimensional subspace.
 Finally, we consider Young tableaux in which all four basis elements appear and two are repeated. In the case that $c$ and $d$ are repeated, we have  the four elements
 \[
\overline{e} \left( \hspace{-1ex}\raisebox{-2ex}{\Young acdbdc } \hspace{-2ex}\right), 
\overline{e} \left( \hspace{-1ex}\raisebox{-2ex}{\Young acdcbd } \hspace{-2ex}\right), 
\overline{e} \left( \hspace{-1ex}\raisebox{-2ex}{\Young bdcbcd } \hspace{-2ex}\right), \quad \text{and} \quad
\overline{e} \left( \hspace{-1ex}\raisebox{-2ex}{\Young adcdbc } \hspace{-2ex}\right). 
\]
As there are $\binom{4}{2}=6$ such choices, this contributes another subspace of dimension $4\cdot 6 = 24$.
  \end{pf}

  We define 
  \[ 
  \cAR_2 := \Sigma^{-5, -1} e(\mathcal{K}^{\sma 6}) = \Sigma^{-5,-1} (\mr{hocolim} \  \{ \mathcal{K}^{\sma 6} \overset{e}{\to} \mathcal{K}^{\sma 6} \overset{e}{\to} \dots   \}), 
  \]
 which is a  split summand of $\Sigma^{-5, -1}\mathcal{K}^{\sma 6}$ as $e$ is an idempotent. We shift the grading by $(-5,-1)$ to make sure that the $\cA^\R(2)$-module generator of  $\rHR^{*,*}(\cAR_2)$ is in $(0,0)$ (see \autoref{lowestD}).  

  \begin{thm} \label{A2free} $\rHR^{*,*}(\cA_2^\R) \cong \cAR(2)$ as an $\cAR(2)$-module. 
  \end{thm}
  
  \begin{proof} By \cite[Corollary 2.7]{BGL},   $\rHR^{*,*}(\cA_2^\R)$ is a free $\cAR(2)$-module
if and only if $\rHR^{*,*}(\cA_2^\R)$ is free as an $\mathbb{M}_2^\R$-module and 
the Margolis homology  $\mathcal{M}( \rHR^{*,*}(\cA_2^\R )\otimes_{\mathbb{M}_2^{\R}} \F, x)$ vanishes for {$x \in \{ \mr{Q}_0^\R, \mr{Q}_1^\R, \overline{\mr{P}}_1^1, \mr{Q}_2^\R, \overline{\mr{P}}_2^1  \}$}, 
  where {$\overline{\mr{P}}_1^1$ and $\overline{\mr{P}}_2^1$ are the elements in $\cAR$ dual to $\upxi_1$ and $\upxi_2$, respectively. }  
    
  Let $\mr{K}_{\R}:= \rHR^{*,*}(\mathcal{K})$. The $\cAR$-module $\rHR^{*,*}(\cA_2^\R)$ is $\bMR$-projective as it is a summand of 
  \[
   \rHR^{*,*}(\Sigma^{-5}\mathcal{K}^{\sma 6}) \cong \Sigma^{-5}\mr{K}_\R^{\otimes_{\bMR} 6 }, 
  \]
  which is $\bMR$-free. However, $\bMR$ is a graded local ring, and over a local ring, being projective is equivalent to being free. Hence,  $\rHR^{*,*}(\cA_2^\R)$ is $\bMR$-free. Since {$\mr{Q}_0^\R, \mr{Q}_1^\R, \mr{Q}_2^\R, \overline{\mr{P}}_1^1,$ and $\overline{\mr{P}}_2^1$} are primitive modulo $(\rho, \tau)$,  {and for $\mr{K} := \mr{K}_\R\otimes_{\bMR} \F$, $i \in \{ 0,1, 2\}$ and $t \in \{ 1,2 \}$}
  \[ \dim_{\F}\mathcal{M}(\mr{K}, \mr{Q}_i^\R) = 2 = \dim_{\F}\mathcal{M}(\mr{K},\overline{\mr{P}}_t^1 ) ,\]
it follows from \autoref{e-dimcount}  that 
\[  \mathcal{M}(\rHR^{*,*}(\cAR_2)\otimes_{\bMR} \F, x)  \cong  \mathcal{M}(\overline{e}(\mr{K}^{\otimes 6}), x) \cong  \overline{e}(\mathcal{M}(\mr{K}, x)^{\otimes 6})   =0  \]
for $x \in \{ {\mr{Q}_0^\R} ,\mr{Q}_1^\R, \overline{\mr{P}}_1^1, \mr{Q}_2^\R, \overline{\mr{P}}_2^1 \}$. 
  Thus, $\rHR^{*,*}(\cA_2^\R)$ is free over $\cAR(2)$.   \autoref{e-dimcount} also implies that the $\bMR$-rank of $\rHR^{*,*}(\cA_2^\R)$ is   64,   and therefore 
   $ \rHR^{*,*}(\cA_2^\R)$ has rank 1 over $\cAR(2)$. 
  \end{proof}
  
  \subsection{An $\R$-motivic lift of $\mathcal{B}(2)$ }
  
   Let  $\tilde{\mr{Q}}_2^\R = [\Sq^4, \mr{Q}_1^\R ]$. Unlike the classical Steenrod algebra, $\mr{Q}_2^\R$ does not agree with $\tilde{\mr{Q}}_2^\R$. Instead, as in \cite{V}*{Example~13.7}, these are related by the formula 
  \[ \mr{Q}_2^\R = [\Sq^4, \mr{Q}_1^\R ] + \rho \Sq^5 \Sq^1. \]
  However, one can check that both $\mr{Q}_2^\R$ and $\tilde{\mr{Q}}_2^\R$ square to zero, hence generate exterior algebras. We define (left) $\cAR(2)$-modules
  \[ \mathcal{B}^\R(2)  :=  \cAR(2) \otimes_{\Lambda(\mr{Q}_2^\R)} \bMR \]
  and
  \begin{equation} \label{defn:B2tilde}
  \tilde{\mathcal{B}}^\R(2)  :=  \cAR(2) \otimes_{\Lambda(\tilde{\mr{Q}}_2^\R)} \bMR. 
  \end{equation}
  Let $\mr{A}_{2}^{\R}$ denote $\rHR^{*,*}(\cAR_2)$.  It is easy to check that the left ideal generated by $\mr{Q}_2^\R$ (likewise $\tilde{\mr{Q}}_2^\R$) in $\cA^\R(2)$ is isomorphic to $\Sigma^{7,3}\mathcal{B}^\R(2)$ (likewise $\Sigma^{7,3}\tilde{\mathcal{B}}^\R(2)$).
  It follows that there is an exact sequence of $\cAR(2)$-modules 
   \begin{equation} \label{exactAB}
  \begin{tikzcd}
  0 \rar & \Sigma^{7,3} \mr{B}_\R \rar[hook, "\iota"]& \mr{A}_2^\R  \rar[two heads, "\pi"] &   \mr{B}_\R \rar & 0,
  \end{tikzcd}
  \end{equation}
 where $\mr{B}_\R$ is either  $\mathcal{B}^\R(2)$ or  $\tilde{ \mathcal{B}}^\R(2)$. The main purpose of this subsection is to show that:
 \begin{lem} \label{AmoduleLift} There exists an exact sequence of $\cAR$-modules whose underlying $\cAR(2)$-module exact sequence is isomorphic to \eqref{exactAB} with  $\mr{B}_\R \cong  \tilde{ \mathcal{B}}^\R(2)$. 
 \end{lem}

 \begin{rmk} \label{notQ2R}
 In the case of  $\mr{B}_\R = \mathcal{B}^\R(2)$, the image of $\Sigma^{7,1} \mathcal{B}^\R(2) \rtarr \mr{A}_2^\R$ is a sub-$\cAR(2)$-module, but not a sub-$\cAR$-module. See \autoref{NotQ2R} for more details.
\end{rmk}

\autoref{AmoduleLift} and \autoref{notQ2R} are direct consequences of the $\cAR$-module structure of $\mr{A}_2^\R$ which can be deduced from the injection 
\[ 
\begin{tikzcd}
\Sigma^{5,1}\mr{A}^\R_2 \rar[hook] & \mr{K}_{\R}^{\otimes_{\bMR} 6 }, 
\end{tikzcd}
\]
 where $\mr{K}_\R = \rHR^{*,*} (\mathcal{K}_\R)$. We do not want to entirely leave this calculation to the reader because, without a few tricks, this calculation is likely to  require computer assistance as $e$ has $ 144$ elements in its expression (in terms of the standard $\F$-generators of $\F[\Sigma_3]$)  and $\mr{K}_{\R}^{\otimes_{\bMR} 6 }$  has $2^{12}$ elements in its $\bMR$-basis. We begin after setting the following notation.
 \begin{notn} Let $x_i$ denote the $\bMR$-generators of $\mr{K}_\R$ in degree $i$ as in \autoref{Kexists}.  
 We use the numbered Young diagram (abbrev. NYD)
 \[ 
\begin{tikzpicture}[scale =.4]
\draw  (0,0) -- (0,3) -- (3,3) -- (3,2) -- (2,2) -- (2,1) -- (1,1) -- (1,0) -- (0,0);
\draw (0,1) -- (1,1) -- (1,2) -- (2,2) -- (2,3);
\draw (0,2) -- (1,2) -- (1,3);
\node[] at (.5,.5) {$i_6$};
\node[] at (.5,1.5) {$i_4$};
\node[] at (.5,2.5) {$i_1$};
\node[] at (1.5,1.5) {$i_5$};
\node[] at (2.5,2.5) {$i_3$};
\node[] at (1.5,2.5) {$i_2$};
\end{tikzpicture}
\]
  to denote the $\bMR$-basis element $x_{i_1} \otimes \dots \otimes x_{i_6} \in  \mr{K}_{\R}^{\otimes_{\bMR} 6 }$, where $i_j \in \{ 0,1,3, 7 \}$.   
 \end{notn} 

As in \autoref{e-dimcount},  let $\overline{\sfR}$ and $\overline{\sfC}$ denote the images of $\sfR$ and $\sfC$ (see \eqref{RCFormula}) in $\F[\Sigma_6]$, respectively.  Since $\overline{e} = \overline{\sfR} \cdot \overline{\sfC}$, we record a few properties of $\overline{\sfR}$  and $\overline{\sfC}$.
Note that $\overline{\sfR}$ annihilates an NYD if it has repeating digits in a row. Likewise,  $\overline{\sfC}$ annihilates an NYD if there are repeating digits in a column. For instance, 
\[ 
\overline{\sfR}(\raisebox{-1em}{\Young{0}{1}{0}{3}{7}{3}}) = 0 = \overline{\sfC}(\raisebox{-1em}{\Young{0}{1}{0}{3}{7}{3}}). 
\]
\begin{rmk} \label{lowestD}
 The lowest degree NYD which is not annihilated by $\overline{e}$ is  
  \[ 
\begin{tikzpicture}[scale =.3]
\draw  (0,0) -- (0,3) -- (3,3) -- (3,2) -- (2,2) -- (2,1) -- (1,1) -- (1,0) -- (0,0);
\draw (0,1) -- (1,1) -- (1,2) -- (2,2) -- (2,3);
\draw (0,2) -- (1,2) -- (1,3);
\node[] at (.5,.5) {3};
\node[] at (.5,1.5) {1};
\node[] at (.5,2.5) {0};
\node[] at (1.5,1.5) {1};
\node[] at (2.5,2.5) {0};
\node[] at (1.5,2.5) {0};
\end{tikzpicture}
\]
 which lives in degree $(5,1)$. Of course, there are multiple NYD's in bidegree $(5,1)$ not annihilated by $\overline{e}$ but their images are the same. Likewise, the NYD of the highest degree not annihilated by $\overline{e}$ is 
 \[ 
\begin{tikzpicture}[scale =.3]
\draw  (0,0) -- (0,3) -- (3,3) -- (3,2) -- (2,2) -- (2,1) -- (1,1) -- (1,0) -- (0,0);
\draw (0,1) -- (1,1) -- (1,2) -- (2,2) -- (2,3);
\draw (0,2) -- (1,2) -- (1,3);
\node[] at (.5,.5) {1};
\node[] at (.5,1.5) {3};
\node[] at (.5,2.5) {7};
\node[] at (1.5,1.5) {3};
\node[] at (2.5,2.5) {7};
\node[] at (1.5,2.5) {7};
\end{tikzpicture}
\]
which lives in bidegree $(28, 11).$
  \end{rmk}
  The  lowest degree element  $ \iota := \overline{e} (\raisebox{-1em}{\Young{0}{0}{0}{1}{1}{3}})$, which serves as the $\cAR$-module generator of $\Sigma^{5,1}\mr{A}_2^\R$, can also be expressed as 
   \[ \iota = \overline{\sfR}( \raisebox{-1em}{\Young{3}{1}{0}{1}{0}{0}} ) \]
  because the other NYDs present in the expression $\overline{\sfC}(\raisebox{-1em}{\Young{0}{0}{0}{1}{1}{3}})$ are annihilated by $\overline{\sfR}$.  
  
 Since the $\R$-motivic Steenrod algebra is cocommutative we {get}
 \[ \overline{\sfR}(\overline{\sfC}(\Sq^i(-))) = \overline{\sfR}(\Sq^i(\overline{\sfC} (-))) = \Sq^i (\overline{\sfR}(\overline{\sfC}(-))).  \]
 This, along with the Cartan formula, allows us to calculate $a \cdot \iota$ for any $a \in \cAR$, fairly easily. For example, 
\begin{eqnarray*}
 \Sq^1 \cdot \iota  &=& \overline{\sfR} (\Sq^1 \raisebox{-1em}{\Young{3}{1}{0}{1}{0}{0}  } )  \\
 &= & \overline{\sfR} ( \raisebox{-1em}{\cYoung{3}{1}{1}{1}{0}{0}  }+ \raisebox{-1em}{\cYoung{3}{1}{0}{1}{1}{0}  }+ \raisebox{-1em}{\Young{3}{1}{0}{1}{0}{1}  }) \\
  &= & \overline{\sfR} (\raisebox{-1em}{\Young{3}{1}{0}{1}{0}{1}  }), \\
  \Sq^2  \cdot \iota  &=& \overline{\sfR} (\Sq^2 \raisebox{-1em}{\Young{3}{1}{0}{1}{0}{0}  } )  \\
 &= & \overline{\sfR} ( \raisebox{-1em}{\cYoung{3}{3}{0}{1}{0}{0}  }+  \raisebox{-1em}{\Young{3}{1}{0}{3}{0}{0}  }   + \tau(\raisebox{-1em}{\cYoung{3}{1}{1}{1}{1}{0}  }+ \raisebox{-1em}{\cYoung{3}{1}{1}{1}{0}{1}  } + \raisebox{-1em}{\cYoung{3}{1}{0}{1}{1}{1}  })   ) \\
  &= & \ \overline{\sfR} (\raisebox{-1em}{\Young{3}{1}{0}{3}{0}{0}  }), \\
\Sq^4 \cdot \iota &=& \overline{\sfR}(\Sq^4 ( \raisebox{-1em}{\Young{3}{1}{0}{1}{0}{0}}) ) \\
&=& \overline{\sfR} (  \raisebox{-1em}{\Young{7}{1}{0}{1}{0}{0}}  + \raisebox{-1em}{\cYoung{3}{3}{0}{3}{0}{0}} + \tau (\raisebox{-1em}{\cYoung{3}{3}{1}{1}{1}{0}} + \raisebox{-1em}{\cYoung{3}{3}{1}{1}{0}{1}} + \raisebox{-1em}{\cYoung{3}{3}{0}{1}{1}{1}} ) \\
&&  \ \ \ \ \ \  \ \ + \tau (\raisebox{-1em}{\cYoung{3}{1}{1}{3}{1}{0}} + \raisebox{-1em}{\cYoung{3}{1}{1}{3}{0}{1}} + \raisebox{-1em}{\Young{3}{1}{0}{3}{1}{1}} ) ) \\
&=&  \overline{\sfR} (  \raisebox{-1em}{\Young{7}{1}{0}{1}{0}{0}}  + \tau \raisebox{-1em}{\Young{3}{1}{0}{3}{1}{1}}).
\end{eqnarray*}
In this way, we calculate  
\begin{eqnarray*}
\tilde{\mr{Q}}_2^\R \cdot \iota &=& \overline{\sfR} (  \raisebox{-1em}{\Young{3}{1}{0}{1}{0}{7}  } + \raisebox{-1em}{\Young{3}{1}{0}{7}{1}{0}  } +\raisebox{-1em}{\Young{7}{3}{1}{1}{0}{0}}  )  \\
\mr{Q}_2^\R \cdot \iota  &=& \overline{\sfR} (  \raisebox{-1em}{\Young{3}{1}{0}{1}{0}{7}  } + \raisebox{-1em}{\Young{3}{1}{0}{7}{1}{0}  } +\raisebox{-1em}{\Young{7}{3}{1}{1}{0}{0}}   + \rho\raisebox{-1em}{\Young{3}{1}{0}{3}{1}{3}} ), \\
\end{eqnarray*}
where  the details are left to the reader. 

\begin{figure}[h]
\includegraphics[width=0.8\textwidth]{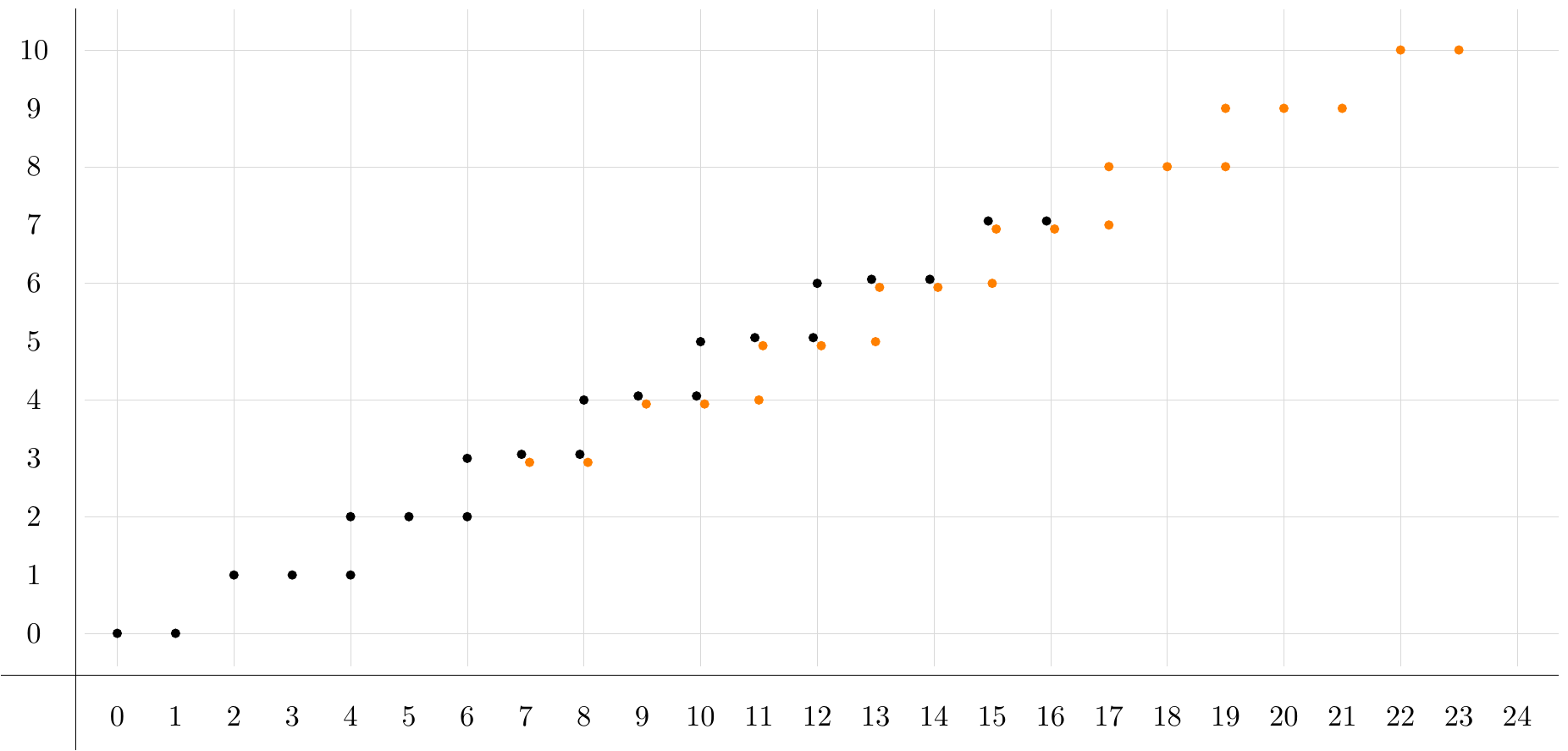}
\caption{$\bMR$-module  generators of $\cAR(2)$. Black dots correspond to generators of  $\tilde{\mathcal{B}}^\R (2)$ and orange dots to $\Sigma^{7,3}\tilde{\mathcal{B}}^\R (2)$.
}
\label{fig:motivicA2gen}
 \end{figure}
 
\begin{rmk} \label{degBA2}
We record (see \autoref{fig:motivicA2gen}), in the notation introduced in \autoref{subsec:wTRT1}, that  
\begin{eqnarray*}
\mathcal{D}_{\tilde{\mathcal{B}}^\R (2)}  = &\{ (0,0), (1,0), (2,1), (3,1), (4,1), (4,2),(5,2),  (6,2), (6,3),  (7,3),  (8,3), (8,4),   \\
&   (9,4), (10,4), (10,5),  (11,5), (12,5), (12,6),  (13,6), (14,6), (15,7), (16,7) \}  
\end{eqnarray*}
and $ \mathcal{D}_{\mr{A}_2^\R} = \{ (i+ 7\epsilon , j + 3 \epsilon): (i,j) \in \mathcal{D}_{\tilde{\mathcal{B}}^\R (2)} \text{ and } \epsilon \in \{ 0,1 \}   \}. $ 
\end{rmk}

\begin{proof}[Proof of \autoref{AmoduleLift}] 
Recall that the image of $\Sigma^{7,3}\tilde{\mathcal{B}}^\R(2)$ in  \eqref{exactAB} is the (left) $\cAR(2)$-submodule of $\mr{A}_2^\R$ generated by $\tilde{\mr{Q}}_2^\R$. We must check that this is
closed under the action of $\cAR$. Since $\Sq^1, \Sq^2, \Sq^4$ are  in $\cAR(2)$, it remains to check that for all $i \geq 3$ and $a \in \cAR(2)$ 
\[ \Sq^{2^i}\cdot (a \tilde{\mr{Q}}_2^\R \cdot \iota) = b \tilde{\mr{Q}}_2^\R \cdot \iota  \]
for some $b \in \cAR(2)$. 
For degree reasons (see \autoref{degBA2}), we only need to consider the case when $i=3$ and $a \in \{ 1, \Sq^1, \Sq^2 \}$. We check 
\begin{eqnarray*}
\Sq^8 \cdot (\tilde{\mr{Q}}_2^\R \cdot \iota) &=& (\Sq^4 \Sq^4 + \Sq^4 \Sq^2 \Sq^2) \tilde{\mr{Q}}_2^\R \cdot \iota\\
\Sq^8 \cdot  (\Sq^1 \tilde{\mr{Q}}_2^\R \cdot \iota) &=& (\Sq^7 \Sq^2 + \Sq^2 \Sq^7) \Sq^1 \tilde{\mr{Q}}_2^\R \cdot \iota \\
\Sq^8 \cdot (\Sq^2 \tilde{\mr{Q}}_2^\R \cdot \iota) &=& (\Sq^4\Sq^4 \Sq^2  + \Sq^4 \Sq^2 \Sq^4 + \tau \Sq^5 \Sq^4 \Sq^1) \Sq^2 \tilde{\mr{Q}}_2^\R \cdot \iota
\end{eqnarray*} 
and thus the result holds. 
\end{proof}

\begin{rmk} \label{NotQ2R}
We notice that  
\[ 
\Sq^8 \cdot (\mr{Q}_2^\R \cdot \iota) =  \overline{\sfR} ( \raisebox{-1em}{ \Young{7}{3}{0}{3}{0}{7} }+ \raisebox{-1em}{\Young{7}{3}{0}{7}{3}{0}} + \tau \raisebox{-1em}{\Young{7}{3}{1}{7}{1}{1}}  
 + \rho(\raisebox{-1em}{\Young{7}{1}{0}{7}{1}{3}} + \raisebox{-1em}{\Young{7}{1}{0}{3}{1}{7}} + \raisebox{-1em}{\Young{3}{1}{0}{7}{1}{7}} ) )
\] 
cannot be equal to $b \mr{Q}_2^\R \cdot \iota$ for any $b \in \cAR(2)$. This is an easy but tedious calculation. 
For the convenience of the reader, we note that an $\F$-basis for the elements in degree $|\Sq^8 | = (8,4)$ of $\cA^\R(2)$ is given by
\[ \{ \Sq^6 \Sq^2, \tau \Sq^7 \Sq^1, \tau \Sq^5 \Sq^2 \Sq^1, \rho \Sq^7, \rho \Sq^6 \Sq^1, \rho \Sq^5 \Sq^2 , \rho \Sq^4 \Sq^2 \Sq^1, \rho^2 \Sq^5 \Sq^1 \}. \]
\end{rmk}

  \subsection{The construction of $\mathcal{Z}_\R$} 
  \
  
 Recall the $\cAR$-module $\tilde{\mr{B}}^\R_2$, as given in \eqref{defn:B2tilde},  and let 
  \[ \mr{B}^\C_2 := \tilde{\mr{B}}^\R_2/(\rho). \] 
  
  \begin{proof}[Proof of \autoref{thm:Z}]
   Since $\mr{B}^\C_2$ is cyclic {as an $\cA^\C$-module}, it admits a May filtration, whose associated graded  is isomorphic to 
  \[ \gr(\mr{B}^\C_{2}) \cong \Lambda ( \upxi_{1,0}, \upxi_{1,1}, \upxi_{1,2}, \upxi_{2,0}, \upxi_{2,1})\]
  and whose $\mr{E}_2$-page of the corresponding May spectral sequence is isomorphic to 
   \begin{equation} \label{MayB}
   {^{\mr{May}}}\mr{E}_{1, \mr{B}^\C_{2} }^{*,*,*,*} \cong  \frac{\bM^\C[\sfh_{i,j}: i \geq 1, j \geq 0] }{(\sfh_{1,0}, \sfh_{1,1},\sfh_{1,2}, \sfh_{2,0}, \sfh_{2,1} )}. 
   \end{equation}
 From this and \autoref{degBA2}, one easily checks that the condition for \autoref{wTRT3} is satisfied.  Thus, there exists $\mathcal{Z}_\R \in \SpfinR$ such that $\rHR^{*,*}(\mathcal{Z}_\R) \cong \tilde{\mr{B}}_2^{\R}$. 
\end{proof}  

\begin{rmk}  Since,  as an $\cA(2)$-module 
 \[ \rH^{*}(\Res{ (\upbeta(\mathcal{Z}_\R))}) \cong \Res_* ( \upbeta_*( \rHR^{*,*}(\mathcal{Z}_\R))) \cong \mathcal{B}(2),\]  the underlying spectrum of $\upbeta(\mathcal{Z}_\R)$ is indeed  one of the spectra $\mathcal{Z}$ considered in  \cite{BE20}, and therefore of type $2$. 
\end{rmk}

\appendix
\section{The $\R$-motivic Adem relations}
\label{AdemSec}

Voevodsky established the motivic version of the Adem relations \cite{V}*{Section~10}. However, his formulas contain some typos, so for the convenience of the reader, we here present the Adem relations, in the $\R$-motivic case.

\begin{prop}\label{RMotivicAdem} 
In the $\R$-motivic Steenrod algebra $\cAR$, the product $\Sq^a \Sq^b$ is equal to 
\begin{enumerate}
\item ($a$ and $b$ both even) 
\[\sum_{j=0}^{a/2} \tau^{j\ \mathrm{mod} 2} \binom{b-1-j}{a-2j}\Sq^{a+b-j}\Sq^j \]
\item ($a$ odd and $b$ even) 
\[\sum_{j=0}^{(a-1)/2}\binom{b-1-j}{a-2j}\Sq^{a+b-j}\Sq^j+\rho\binom{b-j}{a-2j}\Sq^{a+b-j-1}\Sq^j.\]
\item ($a$ even and $b$ odd)
\[\sum_{j=0}^{a/2}\binom{b-1-j}{a-2j}\Sq^{a+b-j}\Sq^j + \rho \binom{b-1-j}{a+1-2j}\Sq^{a+b-j-1}\Sq^{j}.\]
\item ($a$ and $b$ both odd)
\[\sum_{j=0}^{(a-1)/2} \binom{b-1-j}{a-2j}\Sq^{a+b-j}\Sq^j\]
\end{enumerate}
\end{prop}

\begin{rmk}
Given that $\Sq^a = \Sq^1 \Sq^{a-1}$ if $a$ is odd and also that $\Sq^1(\tau) = \rho$, cases (2) and (4) follow from (1) and (3), respectively. Note also that (1) is the classical formula, but with $\tau$ thrown in whenever needed to balance the weights. In formula (2), the left term appears only when $j$ is even, while the second appears only when $j$ is odd. In formula (3), the second term appears only when $j$ is odd.
\end{rmk}

\begin{eg}
Some examples of the $\R$-motivic Adem relation in low degrees are
\[ \Sq^2 \Sq^2 = \tau \Sq^3 \Sq^1 , \qquad \Sq^3 \Sq^2 = \rho \Sq^3 \Sq^1,\]
and
\[\Sq^2 \Sq^3 = \Sq^5 + \Sq^4 \Sq^1 + \rho \Sq^3 \Sq^1.\]
\end{eg}

 \bibliographystyle{amsalpha}

\end{document}